\documentclass[11pt,a4paper]{article}

\usepackage{inputenc}
\usepackage{amsmath}
\usepackage{bm}
\usepackage{bbold}
\usepackage{amsthm}
\usepackage{enumerate}

\usepackage{pseudocode}

\usepackage{hyperref}
\usepackage{breakurl}

\title{Using tropical optimization techniques in bi-criteria decision problems}

\author{N. Krivulin\thanks{Faculty of Mathematics and Mechanics, Saint Petersburg State University, 28 Universitetsky Ave., St.~Petersburg, 198504, Russia, 
nkk@math.spbu.ru.}
\thanks{This work was supported in part by the Russian Foundation for Basic Research (grant No. 18-010-00723)}}

\date{}

\newtheorem{theorem}{Theorem}
\newtheorem{lemma}[theorem]{Lemma}

\theoremstyle{definition}
\newtheorem{example}{Example}

\begin{document}

\maketitle

\begin{abstract}
We consider decision problems of rating alternatives based on their pairwise comparisons according to two criteria. Given pairwise comparison matrices for each criterion, the problem is to find the overall scores of the alternatives. We offer a solution that involves the minimax approximation of the comparison matrices by a common consistent matrix of unit rank in terms of the Chebyshev metric in logarithmic scale. The approximation problem reduces to a bi-objective optimization problem to minimize the approximation errors simultaneously for both comparison matrices. We formulate the problem in terms of tropical (idempotent) mathematics, which focuses on the theory and applications of algebraic systems with idempotent addition. To solve the optimization problem obtained, we apply methods and results of tropical optimization to derive a complete Pareto-optimal solution in a direct explicit form ready for further analysis and straightforward computation. We then exploit this result to solve the bi-criteria decision problem of interest. As illustrations, we present examples of the solution of two-dimensional optimization problems in general form, and of a decision problem with four alternatives in numerical form.
\\

\\

\end{abstract}

\section{Introduction}

Tropical (idempotent) mathematics, which is an area concerned with the theory and applications of algebraic systems with idempotent addition, incorporates tropical optimization as an important research domain. Since the first studies in early 1960s, real-world optimization problems have often served to motivate and illustrate the developments in tropical mathematics. Tropical optimization problems are formulated and solved in the tropical mathematics setting, and appear in many recent works in the area, which include the monographs and textbooks by \cite{Golan2003Semirings,Heidergott2006Maxplus,Mceneaney2006Maxplus,Gondran2008Graphs,Butkovic2010Maxlinear,Maclagan2015Introduction}, and a great many contributed papers.

Applications of tropical optimization cover various problems in project scheduling, location analysis, decision making and in other fields. Some related examples can be found, e.g., in \cite{Krivulin2015Extremal,Krivulin2015Rating,Krivulin2016Using,Krivulin2017Direct,Krivulin2017Tropical,Krivulin2017Using,Krivulin2017Tropicaloptimization}. There are multidimensional tropical optimization problems that can be solved directly to describe all solutions in a compact closed vector form, whereas for other problems, only algorithmic solutions are available, which offer numerical iterative procedures to find a solution if one exists. For a brief overview of tropical optimization problems, one can see, e.g., \cite{Krivulin2014Tropical,Krivulin2015Multidimensional}. 

Multi-criteria decision problems, in which one needs to rate alternatives by evaluating their scores from the results of pairwise comparisons under several criteria, constitute a theoretically interesting and practically important class of problems in decision analysis (see, e.g., \cite{Saaty1990Analytic,Gavalec2015Decision}). The most common solution to the problems is based on the Analytical Hierarchy Process (AHP) method, developed in \cite{Saaty1977Scaling,Saaty1990Analytic,Saaty2013Onthemeasurement}, which involves calculating the principal eigenvectors of pairwise comparison matrices. Available solutions include the fuzzy AHP, interval AHP, and other techniques as in \cite{Laarhoven1983Fuzzy,Gavalec2015Decision,Kubler2016Stateoftheart,Ahn2017Analytic}.

In the context of tropical mathematics, the decision problems of rating alternatives are examined in \cite{Elsner2004Maxalgebra,Elsner2010Maxalgebra,Gursoy2013Analytic,Tran2013Pairwise}, which offer solutions that follow the AHP method with the tropical eigenvectors or subeigenvectors used instead of the conventional principal eigenvectors.

Another approach to the solution of the problems is proposed and developed in \cite{Krivulin2015Rating,Krivulin2016Using,Krivulin2017Tropicaloptimization}, which is based on the minimax log-Chebyshev approximation of pairwise comparison matrices. The approach involves the representation of the approximation problems in terms of tropical mathematics as tropical optimization problems, and the direct solution of these problems using methods and techniques of tropical optimization.

In this paper, we further develop the above approach to solve the problem of rating alternatives from pairwise comparisons under two equally weighted (unweighted) criteria. Given pairwise comparison matrices for each criterion, the problem is to find the overall scores of the alternatives. We offer a solution that involves the minimax approximation of the comparison matrices by a common consistent matrix of unit rank in terms of the Chebyshev metric in logarithmic scale. The approximation problem reduces to a bi-objective optimization problem to minimize the approximation errors simultaneously for both comparison matrices.

Furthermore, we formulate the problem in terms of tropical mathematics as a tropical optimization problem. To solve the optimization problem obtained, we apply methods and results of tropical optimization to derive a complete Pareto-optimal solution in a direct explicit form ready for further analysis and straightforward computation. We then exploit this result to solve the bi-criteria decision problem of interest. As illustrations, we present examples of the solution of two-dimensional optimization problems in general form, and of a decision problem with four alternatives in numerical form.  

The paper is organized as follows. We start in Section~\ref{S-OBODS} with a brief overview of one and bi-criteria decision problems under consideration and their representation as problems of the log-Chebyshev approximation of pairwise comparison matrices. In Section~\ref{S-ADPR}, we give an introduction to basic definitions and notation, and present some preliminary results of tropical algebra to provide an analytical framework for solving a bi-objective tropical optimization problem. Section~\ref{S-UBOOP} includes the main result, which offers a direct complete Pareto-optimal solution to the bi-objective problem. In Section~\ref{S-ETDP}, we illustrate the obtained result with examples of the complete solution of two-dimensional optimization problems in general form. Finally, in Section~\ref{S-ABCDS} we demonstrate application of the solution to a bi-criteria decision problem.

\section{One and bi-criteria decision problems}
\label{S-OBODS}

The method of pairwise comparisons finds wide application in decision making to estimate scores (rates, preferences) of alternatives (choices, decisions) when a direct rating of the alternatives is impossible or infeasible. The method uses the result of pairwise comparisons of alternatives with an appropriate scale under one or several criteria to evaluate the individual score of each alternative   (see, e.g., \cite{Thurstone1927Law,Saaty1977Scaling,Saaty1990Analytic} for further details and application examples).

\subsection{Rating by pairwise comparison}

Consider a problem to rate $n$ alternatives from a given pairwise comparison matrix $\bm{A}=(a_{ij})$, where the entry $a_{ij}$ shows the relative priority of alternative $i$ over $j$. The matrix is symmetrically reciprocal, which implies that the equality $a_{ij}=1/a_{ji}>0$ holds for all $i,j=1,\ldots,n$. A pairwise comparison matrix $\bm{A}$ is called consistent if its entries are transitive in the sense of the condition $a_{ij}=a_{ik}a_{kj}$, which must hold for all $i,j,k=1,\ldots,n$.

Furthermore, each consistent matrix $\bm{A}$ is of unit rank, and has entries $a_{ij}=x_{i}/x_{j}$ given by a positive vector $\bm{x}=(x_{i})$ that entirely specifies $\bm{A}$. If a pairwise comparison matrix $\bm{A}$ is consistent, its related vector $\bm{x}$ defines, up to a positive factor, the individual scores of alternatives. In case that the matrix $\bm{A}$ is inconsistent, as is usually the case in practice, an approximation problem arises to find an approximating consistent matrix $\bm{X}=(x_{ij})$ with $x_{ij}=x_{i}/x_{j}$, or, equivalently, the corresponding vector $\bm{x}=(x_{i})$.

The commonly used approach to the problem is based on the approximation in the spectral norm (see, e.g., \cite{Saaty1977Scaling,Saaty1990Analytic,Saaty2013Onthemeasurement}), which results in the principal (Perron) eigenvector method, where the principal eigenvector of the pairwise comparison matrix is taken as the vector of scores $\bm{x}$. Other solutions, proposed and examined in a range of works, including \cite{Saaty1984Comparison,Barzilai1997Deriving,Chu1998Ontheoptimal,Farkas2003Consistency,Gavalec2015Decision}, employ least squares and logarithmic least squares methods, Chebyshev approximation and some other techniques.

Another approach, which applies the best approximation of matrices in the Chebyshev sense on logarithmic scale, is proposed in \cite{Krivulin2015Rating,Krivulin2016Using}, where the minimax log-Chebyshev approximation is represented as a tropical optimization problem, which can be completely solved in an exact vector form.

\subsection{Minimax approximation in log-Chebyshev sense}

Consider the problem to approximate a pairwise comparison matrix $\bm{A}=(a_{ij})$ by a consistent matrix $\bm{X}=(x_{ij})$, where $a_{ij}=1/a_{ji}$ and $x_{ij}=x_{i}/x_{j}$ for all $i,j=1,\ldots,n$. Observing that the matrix $\bm{A}$ is assumed to be positive, we can measure the approximation error by the Chebyshev distance in logarithmic scale.

Let $\log$ denote a logarithmic function with a base greater than one. Since this function is monotone increasing, the log-Chebyshev distance between $\bm{A}$ and $\bm{X}$ to be minimized can be written as
\begin{equation*}
\max_{1\leq i,j\leq n}|\log a_{ij}-\log x_{ij}|
=
\log\max_{1\leq i,j\leq n}\max\left(\frac{a_{ij}}{x_{ij}},\frac{x_{ij}}{a_{ij}}\right).
\end{equation*}

Moreover, the monotonicity property makes the problem of minimizing the last logarithm equivalent to minimizing its argument. As a result, under the conditions $a_{ij}=1/a_{ji}$ and $x_{ij}=x_{i}/x_{j}$, the approximation problem reduces to minimizing
\begin{equation*}
\max_{1\leq i,j\leq n}\max\left(\frac{a_{ij}}{x_{ij}},\frac{x_{ij}}{a_{ij}}\right)
=
\max_{1\leq i,j\leq n}\max\left(\frac{a_{ij}x_{j}}{x_{i}},\frac{a_{ji}x_{i}}{x_{j}}\right)
=
\max_{1\leq i,j\leq n}\frac{a_{ij}x_{j}}{x_{i}},
\end{equation*}
where the functions on the right are minimized over all positive vectors $\bm{x}=(x_{i})$.

Note that, in approximating reciprocal matrices by consistent matrices of unit rank, minimizing the log-Chebyshev approximation error is equivalent to minimizing the relative error
\begin{equation*}
\max_{1\leq i,j\leq n}\frac{|a_{ij}-x_{i}/x_{j}|}{a_{ij}}.
\end{equation*}

To verify the equivalence (see also \cite{Elsner2004Maxalgebra}), we represent the relative error as
\begin{multline*}  
\max_{1\leq i,j\leq n}\frac{|a_{ij}-x_{i}/x_{j}|}{a_{ij}}
=
\max_{i<j}\max\left\{\left|\frac{x_{i}}{a_{ij}x_{j}}-1\right|,\left|\frac{a_{ij}x_{j}}{x_{i}}-1\right|\right\}
\\
=
\max_{i<j}\max\left\{\frac{x_{i}}{a_{ij}x_{j}}-1,\frac{a_{ij}x_{j}}{x_{i}}-1\right\}
=
\max_{1\leq i,j\leq n}\frac{a_{ij}x_{j}}{x_{i}}
-
1.
\end{multline*}  

Since these error functions differ only by an additive constant, we conclude that both approximation problems with the log-Chebyshev and relative errors are equivalent to the problem of finding vectors $\bm{x}=(x_{i})$ that
\begin{equation*}
\begin{aligned}
&
\text{minimize}
&&
\max_{1\leq i,j\leq n}\frac{a_{ij}x_{j}}{x_{i}}.
\label{P-minmaxaijxjxi}
\end{aligned}
\end{equation*}

Complete solutions to this and related problems in the context of rating alternatives on the basis of pairwise comparisons, are given in \cite{Krivulin2015Rating,Krivulin2016Using,Krivulin2017Tropicaloptimization,Krivulin2018Methods}.

\subsection{Pairwise comparison under two criteria}

Suppose $n$ alternatives are compared in pairs under two equally weighted (unweighted) criteria, which results in two pairwise comparison matrices $\bm{A}=(a_{ij})$ and $\bm{B}=(b_{ij})$. The problem is formulated as a bi-criteria problem to find vectors $\bm{x}=(x_{i})$ such that the matrix $\bm{X}=(x_{i}/x_{j})$ simultaneously approximate both matrices $\bm{A}$ and $\bm{B}$. 

A solution to the problem can be obtained by applying the AHP method (see, e.g., \cite{Saaty1977Scaling,Saaty1990Analytic,Saaty2013Onthemeasurement}). In the weighted case, the AHP solution is based on separate approximation of each matrix $\bm{A}$ and $\bm{B}$ by consistent matrices using their principal eigenvectors. The vector of individual scores of alternatives is calculated as a weighted sum of normalized principal eigenvectors, where the weights (priorities) of the criteria can be found by the principal eigenvector method from a pairwise comparison matrix of criteria, or obtained in other ways. Under the assumption of equal weights, the weighted sum is reduced to the ordinary (unweighted) sum of normalized principal eigenvectors.

In the framework of the minimax log-Chebyshev approximation, the problem can be formulated as the bi-objective problem of finding positive vectors $\bm{x}=(x_{i})$ to
\begin{equation}
\begin{aligned}
&
\text{minimize}
&&
\left(
\max_{1\leq i,j\leq n}\frac{a_{ij}x_{j}}{x_{i}},\ 
\max_{1\leq i,j\leq n}\frac{b_{ij}x_{j}}{x_{i}}
\right).
\label{P-minmaxaijxjxi-maxbijxjxi}
\end{aligned}
\end{equation}

The common way to handle this problem, which has two competing objectives in conflict with each other, is to find a compromise solution that could not be improved. The set of solutions, where one objective can be improved only at the expense of the other objective, is usually considered the best compromise solution, which is called the nondominated or Pareto-optimal solution (see, e.g., \cite{Ehrgott2005Multicriteria,Luc2008Pareto,Pappalardo2008Multiobjective,Benson2009Multiobjective}).  

Note that problem \eqref{P-minmaxaijxjxi-maxbijxjxi} can be solved using the technique described in \cite{Krivulin2015Rating,Krivulin2016Using,Krivulin2017Tropicaloptimization}. The solution involves a matrix $\bm{C}=(c_{ij})$ with the entries $c_{ij}=\max\{a_{ij},b_{ij}\}$ to find vectors $\bm{x}=(x_{i})$ that
\begin{equation*}
\begin{aligned}
&
\text{minimize}
&&
\max_{1\leq i,j\leq n}\frac{c_{ij}x_{j}}{x_{i}}.
\end{aligned}
\end{equation*}

This approach presents an analogue of the AHP decision scheme, which applies the minimax log-Chebyshev approximation instead of the principal eigenvector method and the direct calculation of weighted sums. The solution is based on methods and techniques of tropical optimization, and offers the result in a compact vector form. However, this solution involves a scalarization of the bi-criteria problem, and hence can hardly provide a way to obtain all Pareto-optimal solutions.

Below, we further develop the tropical optimization approach to provide a direct, explicit representation for all Pareto-optimal solutions of problem \eqref{P-minmaxaijxjxi-maxbijxjxi}, which is given in a form ready for further analysis and straightforward computation.

\section{Algebraic definitions and preliminary results}
\label{S-ADPR}

We start with a brief overview of the algebraic definitions and preliminary results of tropical mathematics from \cite{Krivulin2015Extremal,Krivulin2015Multidimensional,Krivulin2017Direct,Krivulin2017Tropical}, which provide an analytical framework for the formulation and solution of the bi-objective tropical optimization problem to be considered in the next section. For further details at both basic and advanced levels, and for application examples, one can consult, e.g., the recent books by \cite{Golan2003Semirings,Heidergott2006Maxplus,Mceneaney2006Maxplus,Gondran2008Graphs,Butkovic2010Maxlinear,Maclagan2015Introduction}.

\subsection{Idempotent semifield}

Consider a nonempty set $\mathbb{X}$ equipped with addition $\oplus$ and multiplication $\otimes$ such that both operations are associative and commutative, addition is idempotent and has zero $\mathbb{0}$, whereas multiplication distributes over addition, has identity $\mathbb{1}$, and is invertible for all nonzero elements. The system $(\mathbb{X},\oplus,\otimes,\mathbb{0},\mathbb{1})$, which is an idempotent commutative semigroup with zero under addition and Abelian group under multiplication, is usually called the idempotent semifield.

Idempotent addition $\oplus$ conforms to the rule $x\oplus x=x$ for all $x\in\mathbb{X}$, and induces a partial order on $\mathbb{X}$ such that $x\leq y$ if and only if $x\oplus y=y$. This partial order is assumed extended to a compatible total order. Invertible multiplication $\otimes$ provides an inverse $x^{-1}$ for any $x\ne\mathbb{0}$ to satisfy the identity $xx^{-1}=\mathbb{1}$. (Here and henceforth, the multiplication symbol $\otimes$ is omitted to save writing.) The powers with integer exponents indicate iterated products, and are defined as $x^{p}=xx^{p-1}$, $x^{-p}=(x^{-1})^{p}$, and $x^{0}=\mathbb{1}$ for all nonzero $x\in\mathbb{X}$ and integer $p\geq1$. Moreover, the equation $x^{p}=a$ is assumed solvable with respect to $x$ for all $a\in\mathbb{X}$ and integer $p\geq1$, which  extends the power notation to rational exponents. 

The operations in the semifield have the following properties with respect to the order relation induced by the idempotent addition. First, the inequalities $x\leq x\oplus y$ and $y\leq x\oplus y$ hold. Furthermore, the inequality $x\oplus y\leq z$ is equivalent to the pair of inequalities $x\leq z$ and $y\leq z$. Both operations $\oplus$ and $\otimes$ are monotone in each argument, which implies that the inequality $x\leq y$ yields the inequalities $x\oplus z\leq y\oplus z$ and $xz\leq yz$ for all $z\in\mathbb{X}$. The inversion is antitone, which means that $x\leq y$ results in $x^{-1}\geq y^{-1}$ for all $x,y\ne\mathbb{0}$. Finally, the exponential inequalities $x^{q}\geq x^{r}$ if $x\leq\mathbb{1}$ and $x^{q}\leq x^{r}$ if $x\geq\mathbb{1}$ are valid under the condition $q\leq r$, where $q$ and $r$ are positive rationals.

As examples of the idempotent semifield under study, consider real semifields $\mathbb{R}_{\max,+}=(\mathbb{R}\cup\{-\infty\},\max,+,-\infty,0)$ and $\mathbb{R}_{\max,\times}=(\mathbb{R}_{+},\max,\times,0,1)$, where $\mathbb{R}$ is the set of reals, and $\mathbb{R}_{+}=\{x\geq0|\ x\in\mathbb{R}\}$. In the semifield $\mathbb{R}_{\max,+}$, which is typically called the max-plus algebra, addition $\oplus$ is defined as $\max$ and multiplication $\otimes$ as $+$. The zero $\mathbb{0}$ and identity $\mathbb{1}$ are given by $-\infty$ and $0$. The inverse $x^{-1}$ coincides with the opposite number $-x$ in the conventional arithmetic. The power $x^{y}$ corresponds to the arithmetic product $yx$, and is well-defined for any $x,y\in\mathbb{R}$. The partial order induced by the idempotent addition is compatible with the natural linear order on $\mathbb{R}$. 

The semifield $\mathbb{R}_{\max,\times}$ is commonly referred to as the max-algebra, and has the operations defined as $\oplus=\max$ and $\otimes=\times$, and neutral elements as $\mathbb{0}=0$ and $\mathbb{1}=1$. The inverse and power notations have the standard meaning, and the partial order extends to the natural linear order. The max-algebra will serve below as the basis for the application of tropical optimization to the bi-criteria decision problem of interest.

\subsection{Matrices and vectors}

We now consider matrices with entries in $\mathbb{X}$, and denote the set of the matrices with $m$ rows and $n$ columns by $\mathbb{X}^{m\times n}$. Idempotent algebra of matrices over $\mathbb{X}$ is routinely defined, where the matrix operations follow the standard rules with the scalar operations $\oplus$ and $\otimes$ in place of the ordinary arithmetic addition and multiplication. Specifically, for any matrices $\bm{A}=(a_{ij})\in\mathbb{X}^{m\times n}$, $\bm{B}=(b_{ij})\in\mathbb{X}^{m\times n}$ and $\bm{C}=(c_{ij})\in\mathbb{X}^{n\times l}$, and a scalar $x\in\mathbb{X}$, matrix addition and multiplication, and scalar multiplication are given by the entry-wise formulas
\begin{equation*}
\{\bm{A}\oplus\bm{B}\}_{ij}
=
a_{ij}\oplus b_{ij},
\qquad
\{\bm{A}\bm{C}\}_{ij}
=
\bigoplus_{k=1}^{n}a_{ik}c_{kj},
\qquad
\{x\bm{A}\}_{ij}
=
xa_{ij}.
\end{equation*}

The properties of the scalar operations $\oplus$ and $\otimes$ with respect to the order relation extend to the matrix operations, where the inequalities are understood entry-wise. A matrix that has all entries equal to $\mathbb{0}$ is the zero matrix denoted by $\bm{0}$. A matrix with at least one nonzero entry in each column is called column-regular.

A square matrix with all diagonal entries equal to $\mathbb{1}$ and the off-diagonal entries to $\mathbb{0}$ is the identity matrix denoted by $\bm{I}$. The power notation serves to indicate repeated multiplication of a matrix with itself, defined as $\bm{A}^{0}=\bm{I}$ and $\bm{A}^{p}=\bm{A}\bm{A}^{p-1}$ for any square matrix $\bm{A}$ and integer $p\geq1$.

Consider a square matrix $\bm{A}=(a_{ij})\in\mathbb{X}^{n\times n}$. The trace of $\bm{A}$ is given by
\begin{equation*}
\mathop\mathrm{tr}\bm{A}
=
a_{11}\oplus\cdots\oplus a_{nn}
=
\bigoplus_{k=1}^{n}a_{kk}.
\end{equation*} 

For any matrices $\bm{A},\bm{B}\in\mathbb{X}^{n\times n}$ and scalar $x\in\mathbb{X}$, the following identities hold:
\begin{equation*}
\mathop\mathrm{tr}(\bm{A}\oplus\bm{B})
=
\mathop\mathrm{tr}\bm{A}
\oplus
\mathop\mathrm{tr}\bm{B},
\qquad
\mathop\mathrm{tr}(\bm{A}\bm{B})
=
\mathop\mathrm{tr}(\bm{B}\bm{A}),
\qquad
\mathop\mathrm{tr}(x\bm{A})
=
x\mathop\mathrm{tr}\bm{A}.
\end{equation*}

To describe the solution to optimization problems in the sequel, we need to define the function, which assigns to the matrix $\bm{A}$ the scalar
\begin{equation}
\mathop\mathrm{Tr}(\bm{A})
=
\mathop\mathrm{tr}\bm{A}
\oplus\cdots\oplus
\mathop\mathrm{tr}\bm{A}^{n}
=
\bigoplus_{k=1}^{n}\mathop\mathrm{tr}\bm{A}^{k}.
\label{E-TrA-trAtrAn}
\end{equation}

Provided that the condition $\mathop\mathrm{Tr}(\bm{A})\leq\mathbb{1}$ holds, we apply the asterate operator, also known as the Kleene star, which yields the matrix
\begin{equation*}
\bm{A}^{\ast}
=
I\oplus\bm{A}\oplus\cdots\oplus\bm{A}^{n-1}
=
\bigoplus_{k=0}^{n-1}\bm{A}^{k}
\end{equation*}

Finally, we use the the spectral radius of the matrix $\bm{A}$, which is given by
\begin{equation}
\lambda
=
\mathop\mathrm{tr}\bm{A}
\oplus\cdots\oplus
\mathop\mathrm{tr}\nolimits^{1/n}(\bm{A}^{n})
=
\bigoplus_{k=1}^{n}\mathop\mathrm{tr}\nolimits^{1/k}(\bm{A}^{k}).
\label{E-lambda-ai1i2ai2i3aiki1}
\end{equation}

A matrix with one column or row forms a vector over $\mathbb{X}$. The vectors are considered as column vectors unless otherwise specified. The set of column vectors with $n$ elements is denoted by $\mathbb{X}^{n}$. A vector with all entries equal to $\mathbb{0}$ is the zero vector denoted by $\bm{0}$. A vector is called regular if it has no zero elements.

For any regular vector $\bm{x}=(x_{i})$, the multiplicative conjugate transposition yields the row vector $\bm{x}^{-}=(x_{i}^{-})$ with the elements $x_{i}^{-}=x_{i}^{-1}$ if $x_{i}\ne\mathbb{0}$, and $x_{i}^{-}=\mathbb{0}$ otherwise.

\subsection{Vector inequalities}

Given a matrix $\bm{A}\in\mathbb{X}^{m\times n}$ and a vector $\bm{d}\in\mathbb{X}^{m}$, consider the problem to solve, with respect to the unknown vector $\bm{x}\in\mathbb{X}^{n}$, the inequality
\begin{equation}
\bm{A}\bm{x}
\leq
\bm{d}.
\label{I-Axleqd}
\end{equation}

A direct solution to the problem is described as follows (see, e.g., \cite{Krivulin2015Extremal}).
\begin{lemma}
\label{L-Axleqd}
For any column-regular matrix $\bm{A}$ and regular vector $\bm{d}$, all solutions to inequality \eqref{I-Axleqd} are given by
\begin{equation*}
\bm{x}
\leq
(\bm{d}^{-}\bm{A})^{-}.
\label{I-xd-A}
\end{equation*}
\end{lemma}

Now suppose that, given a matrix $\bm{A}\in\mathbb{X}^{n\times n}$, the problem is to find regular vectors $\bm{x}\in\mathbb{X}^{n}$ to satisfy the inequality
\begin{equation}
\bm{A}\bm{x}
\leq
\bm{x}.
\label{I-Axleqx}
\end{equation}

The next result, obtained in \cite{Krivulin2015Multidimensional}, offers a direct solution.
\begin{theorem}
\label{T-Axleqx}
For any matrix $\bm{A}$, the following statements hold:
\begin{enumerate}
\item If $\mathop\mathrm{Tr}(\bm{A})\leq\mathbb{1}$, then all regular solutions to \eqref{I-Axleqx} are given by $\bm{x}=\bm{A}^{\ast}\bm{u}$, where $\bm{u}$ is any regular vector.
\item If $\mathop\mathrm{Tr}(\bm{A})>\mathbb{1}$, then there is only the trivial solution $\bm{x}=\bm{0}$.
\end{enumerate}
\end{theorem}

\subsection{Identities and inequalities for traces}

We start with an obvious binomial identity that is valid for any square matrices $\bm{A},\bm{B}\in\mathbb{X}^{n\times n}$ and positive integer $m$ in the following form (see also \cite{Krivulin2017Direct}):
\begin{equation*}
(\bm{A}\oplus\bm{B})^{m}
=
\bm{A}^{m}
\oplus
\bigoplus_{k=1}^{m}\bigoplus_{i_{0}+i_{1}+\cdots+i_{k}=m-k}
\bm{A}^{i_{0}}(\bm{B}\bm{A}^{i_{1}}\cdots\bm{B}\bm{A}^{i_{k}}).
\end{equation*}

Taking trace of both sides and using properties of traces yield
\begin{equation*}
\mathop\mathrm{tr}(\bm{A}\oplus\bm{B})^{m}
=
\mathop\mathrm{tr}\bm{A}^{m}
\oplus
\bigoplus_{k=1}^{m}\bigoplus_{i_{1}+\cdots+i_{k}=m-k}
\mathop\mathrm{tr}(\bm{B}\bm{A}^{i_{1}}\cdots\bm{B}\bm{A}^{i_{k}}).
\end{equation*} 

After summing over $m=1,\ldots,n$, and rearranging terms, we obtain the identity
\begin{equation}
\mathop\mathrm{Tr}(\bm{A}\oplus\bm{B})
=
\bigoplus_{k=1}^{n}\mathop\mathrm{tr}\bm{A}^{k}
\oplus
\bigoplus_{k=1}^{n}
\bigoplus_{m=0}^{n-k}
\bigoplus_{i_{1}+\cdots+i_{k}=m}
\mathop\mathrm{tr}(\bm{B}\bm{A}^{i_{1}}\cdots\bm{B}\bm{A}^{i_{k}}).
\label{E-TrAB}
\end{equation} 

Furthermore, we assume $s,t>\mathbb{0}$, and introduce two functions
\begin{equation}
\begin{aligned}
G(s)
&=
\bigoplus_{k=1}^{n-1}
\bigoplus_{m=1}^{n-k}
\bigoplus_{i_{1}+\cdots+i_{k}=m}
s^{-k/m}\mathop\mathrm{tr}\nolimits^{1/m}(\bm{B}\bm{A}^{i_{1}}\cdots\bm{B}\bm{A}^{i_{k}}),
\\
H(t)
&=
\bigoplus_{k=1}^{n-1}
\bigoplus_{m=1}^{n-k}
\bigoplus_{i_{1}+\cdots+i_{k}=m}
t^{-m/k}\mathop\mathrm{tr}\nolimits^{1/k}(\bm{B}\bm{A}^{i_{1}}\cdots\bm{B}\bm{A}^{i_{k}}).
\end{aligned}
\label{E-Gs-Ht}
\end{equation}

The functions $G$ and $H$ possess a duality property that holds for the pair of inequalities $G(s)\leq t$ and $H(t)\leq s$, which appear to be equivalent. To verify this property, suppose that, for some $s$ and $t$, the first inequality $G(s)\leq t$ is valid. This inequality is equivalent to the system of inequalities
\begin{equation*}
\bigoplus_{i_{1}+\cdots+i_{k}=m}
s^{-k/m}\mathop\mathrm{tr}\nolimits^{1/m}(\bm{B}\bm{A}^{i_{1}}\cdots\bm{B}\bm{A}^{i_{k}})
\leq
t,
\quad
m=1,\ldots,n-k;
\quad
k=1,\ldots,n-1.
\end{equation*}

Multiplication of both sides of the inequalities by $t^{-1}s^{k/m}$, followed by raising to the power $m/k$, leads to the system
\begin{equation*}
\bigoplus_{i_{1}+\cdots+i_{k}=m}
t^{-m/k}\mathop\mathrm{tr}\nolimits^{1/k}(\bm{B}\bm{A}^{i_{1}}\cdots\bm{B}\bm{A}^{i_{k}})
\leq
s,
\quad
m=1,\ldots,n-k;
\quad
k=1,\ldots,n-1.
\end{equation*}

By summing up the inequalities in the system, we obtain the second inequality
\begin{equation*}
H(t)
=
\bigoplus_{k=1}^{n-1}
\bigoplus_{m=1}^{n-k}
\bigoplus_{i_{1}+\cdots+i_{k}=m}
t^{-m/k}\mathop\mathrm{tr}\nolimits^{1/k}(\bm{B}\bm{A}^{i_{1}}\cdots\bm{B}\bm{A}^{i_{k}})
\leq
s.
\end{equation*}

Observing that all transformations performed are invertible, we conclude that both inequalities $G(s)\leq t$ and $H(t)\leq s$ are equivalent.

Let us discuss the computational complexity of the functions $G$ and $H$. Consider the function $G$, and rewrite it in equivalent form as
\begin{equation*}
G(s)
=
\bigoplus_{k=1}^{n-1}
\bigoplus_{m=1}^{n-k}
s^{-k/m}
\mathop\mathrm{tr}\nolimits^{1/m}(\bm{R}_{km}),
\qquad
\bm{R}_{km}
=
\bigoplus_{i_{1}+\cdots+i_{k}=m}
\bm{B}\bm{A}^{i_{1}}\cdots\bm{B}\bm{A}^{i_{k}},
\end{equation*}
which shows that the computational time required to calculate $G$ is determined by the time taken to calculate the matrices $\bm{R}_{km}$ for all $k=1,\ldots,n$ and $m=1,\ldots,n-k$.

Note that $\bm{R}_{km}$ is defined as the sum of products $\bm{B}\bm{A}^{i_{1}}\cdots\bm{B}\bm{A}^{i_{k}}$ over all nonnegative integers $i_{1},\ldots,i_{k}$ such that $i_{1}+\cdots+i_{k}=m$. Since the recurrence equation
\begin{equation*}
\bm{R}_{km}
=
\bm{B}\bm{R}_{k-1,m}\oplus\bm{R}_{k,m-1}\bm{A},
\qquad
\bm{R}_{0m}
=
\bm{A}^{m},
\qquad
\bm{R}_{k0}
=
\bm{B}^{k},
\qquad
\bm{R}_{00}
=
\bm{I},
\end{equation*}
is valid for all $k,m\geq1$, each matrix $\bm{R}_{km}$ can be obtained from the matrices $\bm{R}_{k-1,m}$ and $\bm{R}_{k,m-1}$ through two matrix multiplications and one matrix addition. Observing that one matrix multiplication takes at most $O(n^{3})$ scalar operations, and the number of matrices $\bm{R}_{km}$ used in the coefficients of the function $G$ is $n(n-1)/2$, the overall time needed to calculate $G(s)$ is $O(n^{5})$. The calculation of $H(t)$ takes the same time.

\section{Unconstrained bi-objective optimization problem}
\label{S-UBOOP}

We are now in a position to describe our main result, which offers a complete solution to a tropical bi-objective optimization problem in the form of a direct, explicit representation of both Pareto frontier and related Pareto-optimal solution of the problem. 

Suppose that, given square matrices $\bm{A},\bm{B}\in\mathbb{X}^{n\times n}$, we need to find regular vectors $\bm{x}\in\mathbb{X}^{n}$ that solve the bi-objective optimization problem
\begin{equation}
\begin{aligned}
&
\text{minimize}
&&
(\bm{x}^{-}\bm{A}\bm{x},\ \bm{x}^{-}\bm{B}\bm{x}).
\label{P-minxAxxBx}
\end{aligned}
\end{equation}

To cope with this problem, we implement the approach, which is based on the use of parameters, introduced to represent optimal values of the objective functions in the Pareto frontier. We reduce the problem to a parametrized vector inequality, and then exploit the existence condition for the solution of the inequality to evaluate the parameters and to construct the Pareto frontier. Finally, the solutions of the inequality, which correspond to the parameters in the Pareto frontier, are taken as a complete Pareto-optimal solution to the problem. 

The next statement offers a complete solution to problem \eqref{P-minxAxxBx}.
\begin{theorem}
\label{T-minxAxxBx}
Let $\bm{A}$ be a matrix with spectral radius $\mu>\mathbb{0}$, $\bm{B}$ a matrix with spectral radius $\nu>\mathbb{0}$, and $G(s)$ and $H(t)$ be corresponding functions defined as \eqref{E-Gs-Ht}.

Then the following statements hold:
\begin{enumerate}
\item If $\mu<G(\nu)$, then the Pareto frontier of problem \eqref{P-minxAxxBx} is the set of points $(\alpha,\beta)$ defined by the conditions
\begin{equation}
\mu
\leq
\alpha
\leq
G(\nu),
\qquad
\beta
=
H(\alpha),
\label{I-muleqalphaGnu-E-betaeqHalpha}
\end{equation}
and all regular Pareto-optimal solutions are given by
\begin{equation*}
\bm{x}
=
(\alpha^{-1}\bm{A}
\oplus
\beta^{-1}\bm{B})^{\ast}
\bm{u},
\qquad
\bm{u}
>
\bm{0};
\end{equation*}
\item If $\mu\geq G(\nu)$, then the Pareto frontier is reduced to the single point
\begin{equation*}
\alpha
=
\mu,
\qquad
\beta
=
\nu,
\end{equation*}
and all regular solutions are given by 
\begin{equation*}
\bm{x}
=
(\mu^{-1}\bm{A}
\oplus
\nu^{-1}\bm{B})^{\ast}
\bm{u},
\qquad
\bm{u}
>
\bm{0}.
\end{equation*}
\end{enumerate}
\end{theorem}

\begin{proof}
Denote the minimum values of the objective functions $\bm{x}^{-}\bm{A}\bm{x}$ and $\bm{x}^{-}\bm{B}\bm{x}$ in the Pareto frontier of problem \eqref{P-minxAxxBx} by $\alpha$ and $\beta$. Then, all solutions are defined by the system of equations
\begin{equation*}
\bm{x}^{-}\bm{A}\bm{x}
=
\alpha,
\qquad
\bm{x}^{-}\bm{B}\bm{x}
=
\beta.
\end{equation*}

Since we assume $\alpha$ and $\beta$ to be the minimum values, the set of corresponding regular solutions does not change if the equalities are replaced by the inequalities
\begin{equation*}
\bm{x}^{-}\bm{A}\bm{x}
\leq
\alpha,
\qquad
\bm{x}^{-}\bm{B}\bm{x}
\leq
\beta.
\end{equation*}

By using Lemma~\ref{L-Axleqd}, we solve the first inequality with respect to $\bm{A}\bm{x}$ and the second to $\bm{B}\bm{x}$ to rewrite the system in equivalent form as
\begin{equation*}
\alpha^{-1}\bm{A}\bm{x}
\leq
\bm{x},
\qquad
\beta^{-1}\bm{B}\bm{x}
\leq
\bm{x},
\end{equation*}
which then combine into one inequality
\begin{equation*}
(\alpha^{-1}\bm{A}
\oplus
\beta^{-1}\bm{B})
\bm{x}
\leq
\bm{x}.
\end{equation*}
 
According to Theorem~\ref{T-Axleqx}, regular solutions of the last inequality exist if and only if the following condition holds:
\begin{equation}
\mathop\mathrm{Tr}(\alpha^{-1}\bm{A}\oplus\beta^{-1}\bm{B})
\leq
\mathbb{1},
\label{I-Tralpha1Abeta1B-1}
\end{equation}
and all solutions are given, through a vector of parameters $\bm{u}$, by
\begin{equation}
\bm{x}
=
(\alpha^{-1}\bm{A}
\oplus
\beta^{-1}\bm{B})^{\ast}
\bm{u},
\qquad
\bm{u}
>
\bm{0}.
\label{E-x-alpha1Abeta1Bu}
\end{equation}

To examine the existence condition, we first use \eqref{E-TrAB} for calculating
\begin{equation*}
\mathop\mathrm{Tr}(\alpha^{-1}\bm{A}\oplus\beta^{-1}\bm{B})
=
\bigoplus_{k=1}^{n}\alpha^{-k}\mathop\mathrm{tr}\bm{A}^{k}
\oplus
\bigoplus_{k=1}^{n}
\bigoplus_{m=0}^{n-k}
\bigoplus_{i_{1}+\cdots+i_{k}=m}
\alpha^{-m}\beta^{-k}
\mathop\mathrm{tr}(\bm{B}\bm{A}^{i_{1}}\cdots\bm{B}\bm{A}^{i_{k}}).
\end{equation*} 

In this case, inequality \eqref{I-Tralpha1Abeta1B-1} is equivalent to the system of inequalities
\begin{align*}
\alpha^{-k}\mathop\mathrm{tr}\bm{A}^{k}
&\leq
\mathbb{1},
\\
\beta^{-k}
\bigoplus_{m=0}^{n-k}
\bigoplus_{i_{1}+\cdots+i_{k}=m}
\alpha^{-m}
\mathop\mathrm{tr}(\bm{B}\bm{A}^{i_{1}}\cdots\bm{B}\bm{A}^{i_{k}})
&\leq
\mathbb{1},
\qquad
k=1,\ldots,n.
\end{align*}

By rearranging the terms to isolate powers of $\alpha$ and $\beta$ on the right-hand side, and taking roots, we rewrite the system as
\begin{align*}
\mathop\mathrm{tr}\nolimits^{1/k}(\bm{A}^{k})
&\leq
\alpha,
\\
\bigoplus_{m=0}^{n-k}
\bigoplus_{i_{1}+\cdots+i_{k}=m}
\alpha^{-m/k}
\mathop\mathrm{tr}\nolimits^{1/k}(\bm{B}\bm{A}^{i_{1}}\cdots\bm{B}\bm{A}^{i_{k}})
&\leq
\beta,
\qquad
k=1,\ldots,n.
\end{align*}

We aggregate all these inequalities into two inequalities
\begin{align*}
\alpha
&\geq
\bigoplus_{k=1}^{n}\mathop\mathrm{tr}\nolimits^{1/k}(\bm{A}^{k}),
\\
\beta
&\geq
\bigoplus_{k=1}^{n}\mathop\mathrm{tr}\nolimits^{1/k}(\bm{B}^{k})
\oplus
\bigoplus_{k=1}^{n-1}\bigoplus_{m=1}^{n-k}\bigoplus_{i_{1}+\cdots+i_{k}=m}
\alpha^{-m/k}
\mathop\mathrm{tr}\nolimits^{1/k}(\bm{B}\bm{A}^{i_{1}}\cdots\bm{B}\bm{A}^{i_{k}}).
\end{align*}

With the spectral radii of the matrices $\bm{A}$ and $\bm{B}$ given by
\begin{equation*}
\mu
=
\bigoplus_{k=1}^{n}\mathop\mathrm{tr}\nolimits^{1/k}(\bm{A}^{k}),
\qquad
\nu
=
\bigoplus_{k=1}^{n}\mathop\mathrm{tr}\nolimits^{1/k}(\bm{B}^{k}),
\end{equation*}
and the notation
\begin{equation*}
H(\alpha)
=
\bigoplus_{k=1}^{n-1}\bigoplus_{m=1}^{n-k}\bigoplus_{i_{1}+\cdots+i_{k}=m}
\alpha^{-m/k}
\mathop\mathrm{tr}\nolimits^{1/k}(\bm{B}\bm{A}^{i_{1}}\cdots\bm{B}\bm{A}^{i_{k}}),
\end{equation*}
the last inequalities take the more compact form
\begin{equation}
\alpha
\geq
\mu,
\qquad
\beta
\geq
\nu
\oplus
H(\alpha).
\label{I-alphageqmu-betageqnuHalpha}
\end{equation}

We now consider the feasible area in the $\alpha\beta$-coordinate system, which is defined by these inequalities. Our aim is to determine Pareto-efficient points $(\alpha,\beta)$ in the feasible area, such that no coordinate of the point can be decreased without increasing the other, and thus to construct the Pareto frontier for the problem.

Since all interior points cannot be Pareto-efficient, we need only examine the points on the boundary of the area, which includes an open left vertical segment with $\alpha=\mu$, lower-left segment, where $\beta=\nu\oplus H(\alpha)$, and an open lower horizontal segment with $\beta=\nu$. Observing that the points of both vertical and horizontal segments are obviously not Pareto-efficient, we conclude that the Pareto frontier coincides with the lower-left boundary segment of the feasible area.

To represent the Pareto frontier in a more convenient form, we examine the second inequality at \eqref{I-alphageqmu-betageqnuHalpha}. First, assume that the condition
\begin{equation*}
H(\alpha)
\leq
\nu
\end{equation*} 
is valid. In this case, the Pareto frontier, which is now the lower-left boundary of the area given by the inequalities $\alpha\geq\mu$ and $\beta\geq\nu$, degenerates into one point
\begin{equation*}
\alpha
=
\mu,
\qquad
\beta
=
\nu. 
\end{equation*} 

We solve the inequality $H(\alpha)\leq\nu$ with respect to $\alpha$ by turning to the equivalent inequality $G(\nu)\leq\alpha$. Taking into account that the inequality $\alpha\geq\mu$ holds, we conclude that, under the condition
\begin{equation*}
\mu
\geq
G(\nu),
\end{equation*}
solution \eqref{E-x-alpha1Abeta1Bu}, which simultaneously minimizes both criteria, is reduced to
\begin{equation*} 
\bm{x}
=
(\mu^{-1}\bm{A}
\oplus
\nu^{-1}\bm{B})^{\ast}
\bm{u},
\qquad
\bm{u}
>
\bm{0}.
\end{equation*} 

Otherwise, the Pareto frontier is given by the conditions
\begin{equation*}
\mu
\leq
\alpha
\leq
G(\nu),
\qquad
\beta
=
H(\alpha),
\end{equation*}
whereas the Pareto-optimal solution takes the general form of \eqref{E-x-alpha1Abeta1Bu}.
\qed
\end{proof}

We conclude this section with the observation that the solution obtained has a polynomial time complexity. Indeed, the evaluation of the parameters $\alpha$ and $\beta$ by using the functions $G(s)$ and $H(t)$ requires at most $O(n^{5})$ scalar operations. At the same time, given the parameters $\alpha$ and $\beta$, the calculation of the Kleene star matrix to represent the solution vector $\bm{x}$ takes at most $O(n^{4})$ operations, which leads to the overall linear time complexity of order $O(n^{5})$.

\section{Examples of two-dimensional problems}
\label{S-ETDP}

In this section, we consider illustrative examples of bi-objective two-dimensional problems, and provide complete Pareto-optimal solutions of these problems. The purpose of this section is to demonstrate computational technique involved in the solution, and to give illuminating geometrical illustrations of the results obtained.

\begin{example}
Consider problem \eqref{P-minxAxxBx} with $n=2$ and the matrices
\begin{equation*}
\bm{A}
=
\begin{pmatrix}
a_{11} & a_{12}
\\
a_{21} & a_{22}
\end{pmatrix},
\qquad
\bm{B}
=
\begin{pmatrix}
b_{11} & b_{12}
\\
b_{21} & b_{22}
\end{pmatrix},
\end{equation*}
and assume the entries of both matrices to be nonzero.

To apply Theorem~\ref{T-minxAxxBx}, we first evaluate the spectral radii $\mu$ and $\nu$ of the matrices $\bm{A}$ and $\bm{B}$. With the matrix powers, given by
\begin{equation*}
\bm{A}^{2}
=
\begin{pmatrix}
a_{11}^{2}\oplus a_{12}a_{21} & a_{12}(a_{11}\oplus a_{22})
\\
a_{21}(a_{11}\oplus a_{22}) & a_{12}a_{21}\oplus a_{22}^{2}
\end{pmatrix},
\quad
\bm{B}^{2}
=
\begin{pmatrix}
b_{11}^{2}\oplus b_{12}b_{21} & b_{12}(b_{11}\oplus b_{22})
\\
b_{21}(b_{11}\oplus b_{22}) & b_{12}b_{21}\oplus b_{22}^{2}
\end{pmatrix},
\end{equation*}
we apply \eqref{E-lambda-ai1i2ai2i3aiki1} for $n=2$ to obtain
\begin{equation*}
\mu
=
a_{11}\oplus a_{12}^{1/2}a_{21}^{1/2}\oplus a_{22},
\qquad
\nu
=
b_{11}\oplus b_{12}^{1/2}b_{21}^{1/2}\oplus b_{22}.
\end{equation*}

Furthermore, we use \eqref{E-Gs-Ht} to construct the functions
\begin{equation*}
G(s)
=
s^{-1}\mathop\mathrm{tr}(\bm{B}\bm{A}),
\qquad
H(t)
=
t^{-1}\mathop\mathrm{tr}(\bm{B}\bm{A}),
\end{equation*}
where the trace of the matrix
\begin{equation*}
\bm{B}\bm{A}
=
\begin{pmatrix}
a_{11}b_{11}\oplus a_{12}b_{21} & a_{11}b_{12}\oplus a_{12}b_{22}
\\
a_{21}b_{11}\oplus a_{22}b_{21} & a_{21}b_{12}\oplus a_{22}b_{22}
\end{pmatrix}
\end{equation*}
on the right-hand sides is calculated as
\begin{equation*}
\mathop\mathrm{tr}(\bm{A}\bm{B})
=
a_{11}b_{11}\oplus a_{12}b_{21}
\oplus
a_{21}b_{12}\oplus a_{22}b_{22}.
\end{equation*}

Next, we follow \eqref{E-TrA-trAtrAn} to expand the Kleene star matrix $(\alpha^{-1}\bm{A}\oplus\beta^{-1}\bm{B})^{\ast}$, whose columns generate the solution vectors. Taking onto account that $\alpha\geq\mu\geq a_{ii}$ and $\beta\geq\nu\geq b_{ii}$ for $i=1,2$, we obtain
\begin{equation*}
(\alpha^{-1}\bm{A}\oplus\beta^{-1}\bm{B})^{\ast}
=
\bm{I}
\oplus
\alpha^{-1}\bm{A}\oplus\beta^{-1}\bm{B}
=
\begin{pmatrix}
\mathbb{1} & \alpha^{-1}a_{12}\oplus\beta^{-1}b_{12}
\\
\alpha^{-1}a_{21}\oplus\beta^{-1}b_{21} & \mathbb{1}
\end{pmatrix}.
\end{equation*}

We are now in a position to rewrite the statement of Theorem~\ref{T-minxAxxBx} in terms of the two-dimensional problem under consideration. With the notation $c=\mathop\mathrm{tr}(\bm{B}\bm{A})$, the solution to the problem is given as follows. If $\mu\nu<c$, then the Pareto frontier $(\alpha,\beta)$ of problem \eqref{P-minxAxxBx} is defined by the conditions
\begin{equation}
\mu
\leq
\alpha
\leq
\nu^{-1}c,
\qquad
\beta
=
\alpha^{-1}c.
\label{I-mualphanu1c-betaalpha1c}
\end{equation}

All regular Pareto-optimal solutions are written, using a vector $\bm{u}=(u_{1},u_{2})^{T}$, as
\begin{equation}
\bm{x}
=
\left(
\begin{array}{cc}
\mathbb{1} & \alpha^{-1}a_{12}\oplus\beta^{-1}b_{12}
\\
\alpha^{-1}a_{21}\oplus\beta^{-1}b_{21} & \mathbb{1}
\end{array}
\right)
\bm{u},
\qquad
\bm{u}
>
\bm{0}.
\label{E-x1alpha1a12beta1b12u-u0}
\end{equation}

Let us examine the collinearity of the columns in the generating matrix, given by
\begin{equation*}
\begin{pmatrix}
\mathbb{1}
\\
\alpha^{-1}a_{21}\oplus\beta^{-1}b_{21}
\end{pmatrix},
\qquad
\begin{pmatrix}
\alpha^{-1}a_{12}\oplus\beta^{-1}b_{12}
\\
\mathbb{1}
\end{pmatrix}.
\end{equation*}

It is easy to see that these two columns are collinear if and only if the equality condition $(\alpha^{-1}a_{12}\oplus\beta^{-1}b_{12})(\alpha^{-1}a_{21}\oplus\beta^{-1}b_{21})=\mathbb{1}$ holds. After expanding the left-hand side, the condition becomes
\begin{equation}
\alpha^{-2}a_{12}a_{21}
\oplus
\alpha^{-1}\beta^{-1}
(a_{12}b_{21}
\oplus
a_{21}b_{12})
\oplus
\beta^{-2}b_{12}b_{21}
=
\mathbb{1}.
\label{E-alpha2a12a22}
\end{equation}

It follows from the conditions on $\alpha$ and $\beta$ at \eqref{I-mualphanu1c-betaalpha1c} that $\alpha\geq\mu\geq(a_{12}a_{21})^{1/2}$, $\alpha\beta=c\geq a_{12}b_{21}\oplus a_{21}b_{12}$, and $\beta=\alpha^{-1}c\geq\nu\geq(b_{12}b_{21})^{1/2}$.

As a direct consequence, we obtain the following inequalities: $\alpha^{-2}a_{12}a_{21}\leq\mathbb{1}$, $\alpha^{-1}\beta^{-1}(a_{12}b_{21}\oplus a_{21}b_{12})\leq\mathbb{1}$, and $\beta^{-2}b_{12}b_{21}\leq\mathbb{1}$. Note that the first inequality holds as equality if and only if $\alpha=\mu$ and $\mu=(a_{12}a_{21})^{1/2}$. The second inequality becomes an equality if $c=a_{12}b_{21}\oplus a_{21}b_{12}$, and the third does if $\alpha=\nu^{-1}c$ and $\nu=(b_{12}b_{21})^{1/2}$.

The inequalities obtained combine into one inequality, which has the same left-hand side as the collinearity condition at \eqref{E-alpha2a12a22},
\begin{equation}
\alpha^{-2}a_{12}a_{21}
\oplus
\alpha^{-1}\beta^{-1}
(a_{12}b_{21}
\oplus
a_{21}b_{12})
\oplus
\beta^{-2}b_{12}b_{21}
\leq
\mathbb{1}.
\label{I-alpha2a12a22}
\end{equation}

This composite inequality holds as equality, resulting in collinear columns in the generating matrix, if and only if at least one component inequality holds as equality.

We now summarize the above discussion on the collinearity of columns to refine the solution by dropping one of the columns, say the first, if they are collinear. We consider the following conditions, which yield collinear columns. 

First suppose that the condition $\alpha=\mu=(a_{12}a_{21})^{1/2}$ holds, Then, $\beta=\mu^{-1}c$, and the Pareto-optimal solution is given up to a positive factor by the vector
\begin{equation*}
\begin{pmatrix}
a_{12}^{1/2}a_{21}^{-1/2}\oplus a_{12}^{1/2}a_{21}^{1/2}b_{12}c^{-1}
\\
\mathbb{1}
\end{pmatrix}.
\end{equation*}

Under the condition $c=a_{12}b_{21}\oplus a_{21}b_{12}$, we have $\beta=\alpha^{-1}(a_{12}b_{21}\oplus a_{21}b_{12})$, and the solutions take the form of the vector
\begin{equation*}
\begin{pmatrix}
\alpha^{-1}a_{12}\oplus\alpha b_{12}(a_{12}b_{21}\oplus a_{21}b_{12})^{-1}
\\
\mathbb{1}
\end{pmatrix},
\qquad
\mu
\leq
\alpha
\leq
\nu^{-1}(a_{12}b_{21}\oplus a_{21}b_{12}).
\end{equation*}

Note that these solutions are generated by the matrix with columns obtained from the above vector by setting $\alpha=\mu$ and $\alpha=\nu^{-1}(a_{12}b_{21}\oplus a_{21}b_{12})$, which is given by 
\begin{equation*}
\begin{pmatrix}
\mu^{-1}a_{12}\oplus\mu b_{12}(a_{12}b_{21}\oplus a_{21}b_{12})^{-1}
&\;\;&
\nu a_{12}(a_{12}b_{21}\oplus a_{21}b_{12})^{-1}\oplus\nu^{-1}b_{12}
\\
\mathbb{1}
& &
\mathbb{1}
\end{pmatrix}.
\end{equation*}

Furthermore, with $\beta=\nu=(b_{12}b_{21})^{1/2}$, we have $\alpha=(b_{12}b_{21})^{-1/2}c$, which yields the solution
\begin{equation*}
\begin{pmatrix}
a_{12}b_{12}^{1/2}b_{21}^{1/2}c
\oplus
b_{12}^{-1/2}b_{21}^{1/2}
\\
\mathbb{1}
\end{pmatrix}.
\end{equation*}

If none of the above conditions is valid, inequality \eqref{I-alpha2a12a22} becomes strict, which means that the columns in the generating matrix are non-collinear. In this case, the solution retains the general form \eqref{E-x1alpha1a12beta1b12u-u0}, where the matrix cannot be reduced to a vector.

Consider the Pareto frontier, which forms a segment with the endpoints
\begin{equation*}
(\mu,\ \mu^{-1}c),
\qquad
(\nu^{-1}c,\ \nu).
\end{equation*}

The first endpoint corresponds to the solution
\begin{equation*}
\bm{x}_{\mu}
=
\begin{pmatrix}
\mathbb{1} & \mu^{-1}a_{12}
\oplus
\mu b_{12}c^{-1}
\\
\mu^{-1}a_{21}
\oplus
\mu b_{21}c^{-1} & \mathbb{1}
\end{pmatrix}
\bm{u},
\qquad
\bm{u}
>
\bm{0}.
\end{equation*}

It follows from the above discussion that, if either the conditions $\mu=(a_{12}a_{21})^{1/2}$ or $c=a_{12}b_{21}\oplus a_{21}b_{12}$ are satisfied, then the solution takes the form of unique (up to a positive factor) vectors, given respectively by
\begin{equation*}
\begin{pmatrix}
a_{12}^{1/2}a_{21}^{-1/2}
\oplus
a_{12}^{1/2}a_{21}^{1/2}b_{12}c^{-1}
\\
\mathbb{1}
\end{pmatrix},
\qquad
\begin{pmatrix}
\mu^{-1}a_{12}
\oplus
\mu b_{12}(a_{12}b_{21}\oplus a_{21}b_{12})^{-1}
\\
\mathbb{1}
\end{pmatrix}.
\end{equation*}

The second endpoint yields the solution in the form
\begin{equation*}
\bm{x}_{\nu}
=
\begin{pmatrix}
\mathbb{1} & \nu a_{12}c^{-1}
\oplus
\nu^{-1}b_{12}
\\
\nu a_{21}c^{-1}
\oplus
\nu^{-1}b_{21} & \mathbb{1}
\end{pmatrix}
\bm{v},
\qquad
\bm{v}
>
\bm{0},
\end{equation*}
which, under the conditions $c=a_{12}b_{21}\oplus a_{21}b_{12}$ or $\nu=(b_{12}b_{21})^{1/2}$, reduces to unique vectors, defined respectively as
\begin{equation*}
\begin{pmatrix}
\nu a_{12}(a_{12}b_{21}\oplus a_{21}b_{12})^{-1}
\oplus
\nu^{-1}b_{12}
\\
\mathbb{1}
\end{pmatrix},
\qquad
\begin{pmatrix}
a_{12}b_{12}^{1/2}b_{21}^{1/2}c^{-1}
\oplus
b_{12}^{1/2}b_{21}^{-1/2}
\\
\mathbb{1}
\end{pmatrix}.
\end{equation*}

In the case that $\mu\nu\geq c$, the Pareto frontier shrinks into the single point
\begin{equation*}
\alpha
=
\mu,
\qquad
\beta
=
\nu,
\end{equation*}
and all regular solutions are given by 
\begin{equation*}
\bm{x}
=
\begin{pmatrix}
\mathbb{1} & \mu^{-1}a_{12}\oplus\nu^{-1}b_{12}
\\
\mu^{-1}a_{21}\oplus\nu^{-1}b_{21} & \mathbb{1}
\end{pmatrix}
\bm{u},
\qquad
\bm{u}
>
\bm{0}.
\end{equation*}

If one of the conditions $\mu=(a_{12}a_{21})^{1/2}$, $c=a_{12}b_{21}\oplus a_{21}b_{12}$, and $\nu=(b_{12}b_{21})^{1/2}$ holds, the solution reduces to unique (up to a positive factor) vectors, given by
\begin{equation*}
\begin{pmatrix}
a_{12}^{1/2}a_{21}^{-1/2}
\oplus
\nu^{-1}b_{12}
\\
\mathbb{1}
\end{pmatrix},
\qquad
\begin{pmatrix}
\mu^{-1}a_{12}\oplus\nu^{-1}b_{12}
\\
\mathbb{1}
\end{pmatrix},
\qquad
\begin{pmatrix}
\mu^{-1}a_{12}
\oplus
b_{12}^{1/2}b_{21}^{-1/2}
\\
\mathbb{1}
\end{pmatrix}.
\end{equation*}

Note that, in the case when the Pareto frontier degenerates into one point, the solution does not have to be a unique (up to a positive factor) vector, and can be a cone formed by two non-collinear vectors, if the above conditions do not hold.  

In Fig.~\ref{F-EPF} and \ref{F-EPOSCVS}, we provide a graphical illustration of the discussion, given in the framework of the $\mathbb{R}_{\max,\times}$ semifield (max-algebra). Fig.~\ref{F-EPF} offers examples of Pareto frontiers in the form of a segment (left), and in a degenerate form of a point (right).
\begin{figure}[ht]
\setlength{\unitlength}{1mm}
\begin{center}
\begin{picture}(55,55)

\put(0,5){\vector(1,0){52}}
\put(5,0){\vector(0,1){52}}

\put(4,15){\line(1,0){41}}
\put(25,4){\line(0,1){41}}

\put(35,15){\thicklines\line(1,0){10}}
\multiput(36,15)(0.8,0){11}{\line(0,1){0.8}}

\put(25,20){\thicklines\line(0,1){25}}
\multiput(25,21)(0,0.8){30}{\line(1,0){0.8}}

\qbezier(12.5,45)(16,16)(45,12.5)


{\thicklines{\qbezier(25,20)(27,18)(35,14.9)}}


\qbezier[14](25.1,20.1)(27.1,18.1)(35.1,15.1)
\qbezier[14](25.15,20.15)(27.15,18.15)(35.15,15.15)
\qbezier[14](25.2,20.2)(27.2,18.2)(35.2,15.2)
\qbezier[14](25.25,20.25)(27.25,18.25)(35.25,15.25)
\qbezier[14](25.3,20.3)(27.3,18.3)(35.3,15.3)
\qbezier[14](25.35,20.35)(27.35,18.35)(35.35,15.35)
\qbezier[14](25.4,20.4)(27.4,18.4)(35.4,15.4)
\qbezier[14](25.45,20.45)(27.45,18.45)(35.45,15.45)
\qbezier[14](25.5,20.5)(27.5,18.5)(35.5,15.5)
\qbezier[14](25.55,20.55)(27.55,18.55)(35.55,15.55)
\qbezier[14](25.6,20.6)(27.6,18.6)(35.6,15.6)
\qbezier[14](25.65,20.65)(27.65,18.65)(35.65,15.65)

\put(4,20){\line(1,0){21}}
\put(35,4){\line(0,1){11}}

\put(25,20){\circle*{1.5}}
\put(35,15){\circle*{1.5}}

\put(2,2){$0$}

\put(1,15){$\nu$}
\put(24,1){$\mu$}

\put(0,21){$c/\mu$}

\put(33,1){$c/\nu$}


\put(8,47){$\alpha\beta=c$}

\put(1,52){$\beta$}
\put(52,1){$\alpha$}

\end{picture}
\hspace{7\unitlength}
\begin{picture}(55,55)

\put(0,5){\vector(1,0){52}}
\put(5,0){\vector(0,1){52}}

\put(4,25){\line(1,0){41}}
\put(20,4){\line(0,1){41}}

\put(20,25){\thicklines\line(1,0){25}}
\multiput(21,25)(0.8,0){30}{\line(0,1){0.8}}

\put(20,25){\thicklines\line(0,1){20}}
\multiput(20,25)(0,0.8){25}{\line(1,0){0.8}}


\qbezier(10,45)(15,15)(45,10)


\put(20,25){\circle*{1.5}}

\put(2,2){$0$}

\put(1,25){$\nu$}
\put(19,1){$\mu$}


\put(39,13){$\alpha\beta=c$}

\put(1,52){$\beta$}
\put(52,1){$\alpha$}

\end{picture}
\end{center}
\caption{Examples of Pareto frontiers with $\mu\nu<c$ (left), and $\mu\nu\geq c$ (right), where $c=\mathop\mathrm{tr}(\bm{B}\bm{A})$.}
\label{F-EPF}
\end{figure}

In Fig.~\ref{F-EPOSCVS}, we demonstrate examples of the Pareto-optimal solutions in the form of a cone (left), and of a vector (right).
\begin{figure}[ht]
\setlength{\unitlength}{1mm}
\begin{center}
\begin{picture}(55,55)

\put(5,0){\vector(0,1){52}}
\put(0,5){\vector(1,0){52}}

\put(4,15){\line(1,0){31}}

\put(12.5,4){\line(0,1){11}}
\put(35,4){\line(0,1){11}}

\put(5,5){\line(3,4){24}}
\multiput(5,5)(0.6,0.8){40}{\line(1,0){1}}
\put(5,5){\thicklines\vector(3,4){7.3}}

\put(5,5){\line(3,1){40}}
\multiput(5,5)(1.0,0.3333){40}{\line(0,1){1}}
\put(5,5){\thicklines\vector(3,1){30}}

\put(5,5){\thicklines\vector(4,3){33}}

\put(2,2){$0$}
\put(2,15){$1$}

\put(7,1){$a_{12}/\mu\oplus\mu b_{12}/c$}
\put(30,1){$\nu a_{12}/c\oplus b_{12}/\nu$}

\put(1,52){$x_{2}$}
\put(52,1){$x_{1}$}

\put(9,17){$\bm{x}_{\mu}$}
\put(37,12){$\bm{x}_{\nu}$}

\put(36,32){$\bm{x}$}

\end{picture}
\hspace{7\unitlength}
\begin{picture}(55,55)

\put(5,0){\vector(0,1){52}}
\put(0,5){\vector(1,0){52}}

\put(5,5){\line(2,1){40}}
\put(5,5){\thicklines\vector(2,1){30}}

\put(4,15){\line(1,0){21}}
\put(25,4){\line(0,1){11}}



\put(2,2){$0$}
\put(2,15){$1$}

\put(16.5,1){$a_{12}/\mu\oplus b_{12}/\nu$}

\put(1,52){$x_{2}$}
\put(52,1){$x_{1}$}


\put(34,22){$\bm{x}$}

\end{picture}
\end{center}
\caption{Examples of Pareto-optimal solution cone (left), and single-vector solution (right).}
\label{F-EPOSCVS}
\end{figure}
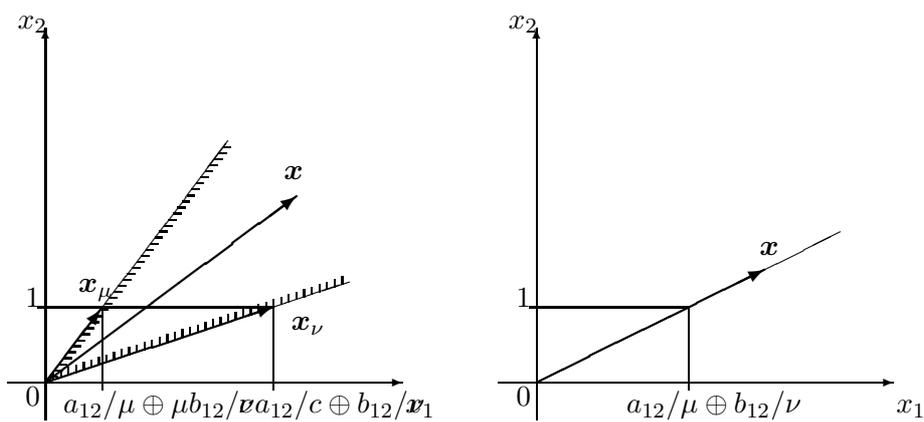
\end{example}

\begin{example}
Suppose that the matrices in problem \eqref{P-minxAxxBx} are symmetrically reciprocal, and given by
\begin{equation*}
\bm{A}
=
\begin{pmatrix}
\mathbb{1} & a
\\
a^{-1} & \mathbb{1}
\end{pmatrix},
\qquad
\bm{B}
=
\begin{pmatrix}
\mathbb{1} & b
\\
b^{-1} & \mathbb{1}
\end{pmatrix},
\qquad
a\ne b.
\end{equation*}

By using results of the previous example, we obtain
\begin{equation*}
\mu
=
\nu
=
\mathbb{1},
\qquad
\mathop\mathrm{tr}(\bm{A}\bm{B})
=
ab^{-1}
\oplus
a^{-1}b
=
c.
\end{equation*}

Furthermore, we represent the matrix, which generates the solutions, as follows:
\begin{equation*}
(\alpha^{-1}\bm{A}\oplus\beta^{-1}\bm{B})^{\ast}
=
\begin{pmatrix}
\mathbb{1} & \alpha^{-1}a\oplus\beta^{-1}b
\\
\alpha^{-1}a^{-1}\oplus\beta^{-1}b^{-1} & \mathbb{1}
\end{pmatrix}.
\end{equation*}

Since $a\ne b$, we have $c=a^{-1}b\oplus ab^{-1}>\mathbb{1}$, and hence the condition $\mu\nu<\mathop\mathrm{tr}(\bm{A}\bm{B})$ holds. The Pareto frontier $(\alpha,\beta)$ is defined as
\begin{equation*}
\mathbb{1}
\leq
\alpha
\leq
c,
\qquad
\beta
=
\alpha^{-1}c,
\end{equation*}
and all regular Pareto-optimal solutions are given by
\begin{equation*}
\bm{x}
=
\begin{pmatrix}
\mathbb{1} & \alpha^{-1}a\oplus\beta^{-1}b
\\
\alpha^{-1}a^{-1}\oplus\beta^{-1}b^{-1} & \mathbb{1}
\end{pmatrix}
\bm{u},
\qquad
\bm{u}
>
\bm{0}.
\end{equation*}

Observing that $\beta=\alpha^{-1}c$, $c=a^{-1}b\oplus ab^{-1}$ and $\alpha\leq c$, we have
\begin{equation*}
(\alpha^{-1}a\oplus\beta^{-1}b)(\alpha^{-1}a^{-1}\oplus\beta^{-1}b^{-1})
=
\alpha^{-2}\oplus(ab^{-1}\oplus a^{-1}b)c^{-1}\oplus\alpha^{-2}c^{-2}
=
\mathbb{1},
\end{equation*}
which implies that the columns in the generating matrix are collinear. Taking one of them, say the second, we reduce the generating matrix to one parametrized column
\begin{equation*}
\begin{pmatrix}
\alpha^{-1}a
\oplus
\alpha bc^{-1}
\\
\mathbb{1}
\end{pmatrix},
\qquad
\mathbb{1}
\leq
\alpha
\leq
c.
\end{equation*}

Then, we see that, as $\alpha$ passes from $\mathbb{1}$ to $c$, the value of $\alpha^{-1}a\oplus\alpha bc^{-1}$ changes from $a$ to $b$. As a result, all solutions can be generated by the columns $(a,\mathbb{1})^{T}$ and $(b,\mathbb{1})^{T}$ to provide a new matrix representation, where the parameter $\alpha$ is eliminated,   
\begin{equation*}
\bm{x}
=
\begin{pmatrix}
a & b
\\
\mathbb{1} & \mathbb{1}
\end{pmatrix}
\bm{u},
\qquad
\bm{u}
>
\bm{0}.
\end{equation*}

The endpoints of the frontier segment are given by
\begin{equation*}
(\mathbb{1},\ c),
\qquad
(c,\ \mathbb{1}),
\end{equation*}
which correspond to the solutions
\begin{equation*}
\bm{x}_{\mu}
=
\begin{pmatrix}
a
\\
\mathbb{1}
\end{pmatrix}
u,
\qquad
u>0;
\qquad
\bm{x}_{\nu}
=
\begin{pmatrix}
b
\\
\mathbb{1}
\end{pmatrix}
v,
\qquad
v>0.
\end{equation*}

In Fig.~\ref{F-EPFSCSC}, we give a graphical illustration in terms of the semifield $\mathbb{R}_{\max,\times}$ for the example considered. The Pareto frontier is shown on the left as the lower-left segment of the hatched border of the feasible area for the criteria. The Pareto-optimal solutions $\bm{x}$ are depicted on the right as the cone with hatched boundaries, generated by the vectors $\bm{x}_{\mu}$ and $\bm{x}_{\nu}$.
\begin{figure}[ht]
\setlength{\unitlength}{1mm}
\begin{center}
\begin{picture}(55,55)

\put(0,5){\vector(1,0){52}}
\put(5,0){\vector(0,1){52}}

\put(4,15){\line(1,0){41}}
\put(15,4){\line(0,1){41}}

\put(35,15){\thicklines\line(1,0){10}}
\multiput(36,15)(0.8,0){11}{\line(0,1){0.8}}

\put(15,35){\thicklines\line(0,1){10}}
\multiput(15,36)(0,0.8){11}{\line(1,0){0.8}}

{\thicklines{\qbezier(15,35)(20,20)(35,15)}}

\qbezier[40](15.1,35.1)(20.1,20.1)(35.1,15.1)
\qbezier[40](15.15,35.15)(20.15,20.15)(35.15,15.15)
\qbezier[40](15.2,35.2)(20.2,20.2)(35.2,15.2)
\qbezier[40](15.25,35.25)(20.25,20.25)(35.25,15.25)
\qbezier[40](15.3,35.3)(20.3,20.3)(35.3,15.3)
\qbezier[40](15.35,35.35)(20.35,20.35)(35.35,15.35)
\qbezier[40](15.4,35.4)(20.4,20.4)(35.4,15.4)
\qbezier[40](15.45,35.45)(20.45,20.45)(35.45,15.45)
\qbezier[40](15.5,35.5)(20.5,20.5)(35.5,15.5)
\qbezier[40](15.55,35.55)(20.55,20.55)(35.55,15.55)
\qbezier[40](15.6,35.6)(20.6,20.6)(35.6,15.6)
\qbezier[40](15.65,35.65)(20.65,20.65)(35.65,15.65)

\put(4,35){\line(1,0){11}}
\put(35,4){\line(0,1){11}}

\put(15,35){\circle*{1.5}}
\put(35,15){\circle*{1.5}}

\put(2,2){$0$}

\put(2,15){$1$}
\put(14.3,1){$1$}

\put(1,35){$c$}

\put(34.3,1){$c$}

\put(32,32){$c=a^{-1}b\oplus ab^{-1}$}

\put(1,52){$\beta$}
\put(52,1){$\alpha$}

\end{picture}
\hspace{7\unitlength}
\begin{picture}(55,55)

\put(0,5){\vector(1,0){52}}
\put(5,0){\vector(0,1){52}}

\put(4,15){\line(1,0){21}}

\put(11.75,4){\line(0,1){11}}
\put(25,4){\line(0,1){11}}

\put(5,5){\line(2,3){25}}
\multiput(5,5)(0.5,0.75){50}{\line(1,0){1}}
\put(5,5){\thicklines\vector(2,3){6.6}}

\put(5,5){\line(2,1){40}}
\multiput(5,5)(0.8,0.4){50}{\line(0,1){1}}
\put(5,5){\thicklines\vector(2,1){20}}

\put(5,5){\thicklines\vector(4,3){35}}

\put(2,2){$0$}
\put(2,15){$1$}

\put(11,1){$a$}
\put(24.5,1){$b$}

\put(1,52){$x_{2}$}
\put(52,1){$x_{1}$}

\put(8,17){$\bm{x}_{\mu}$}
\put(27,13){$\bm{x}_{\nu}$}

\put(38,33){$\bm{x}$}

\end{picture}
\end{center}
\caption{Examples of Pareto frontier segment (left) and corresponding solution cone (right).}
\label{F-EPFSCSC}
\end{figure}
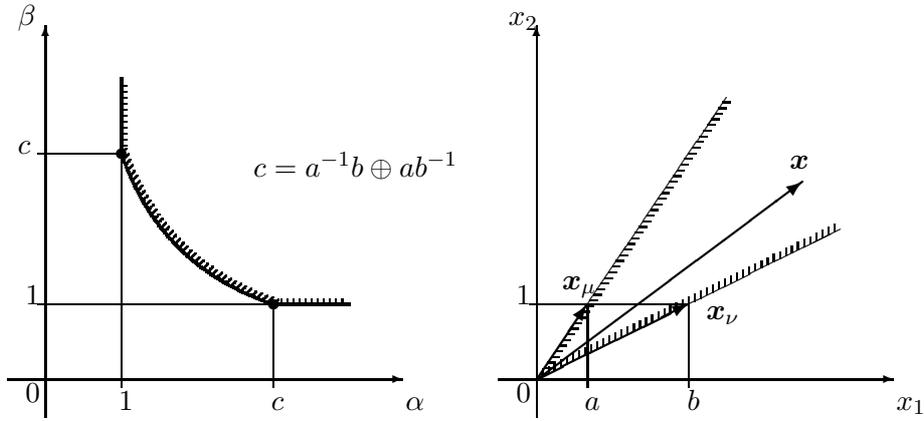

It is clear from the illustration that, as $b$ tends to $a$, the frontier degenerates into the point $(1,1)$, whereas the solution cone does to the single (up to a positive factor) vector $\bm{x}=(a,1)^{T}$.
\end{example}

\section{Application to bi-criteria decision problem}
\label{S-ABCDS}

We now turn back to the problem of evaluating scores of alternatives based on pairwise comparisons under two equally weighted (unweighted) criteria, and apply results of the previous sections to solve this problem. We use the solution approach, which involves minimax log-Chebyshev approximation of pairwise comparison matrices, and leads to the solution of the bi-objective problem in the form of \eqref{P-minmaxaijxjxi-maxbijxjxi}.

It is not difficult to see that the representation of problem \eqref{P-minmaxaijxjxi-maxbijxjxi} in terms of the semifield $\mathbb{R}_{\max,\times}$ (max-algebra) yields problem \eqref{P-minxAxxBx}. In this case, a complete Pareto-optimal solution to the problem is given by the direct application of Theorem~\ref{T-minxAxxBx}. Below, we demonstrate the use of the theorem by an illustrative numerical example.

\begin{example}
Consider a problem to rate $n=4$ alternatives using pairwise comparison data obtained according to two equally weighted criteria, and given by the matrices
\begin{equation*}
\bm{A}
=
\begin{pmatrix}
1 & 3 & 4 & 2
\\
1/3 & 1 & 1/2 & 1/3
\\
1/4 & 2 & 1 & 4
\\
1/2 & 3 & 1/4 & 1
\end{pmatrix},
\qquad
\bm{B}
=
\begin{pmatrix}
1 & 2 & 4 & 2
\\
1/2 & 1 & 1/3 & 1/2
\\
1/4 & 3 & 1 & 4
\\
1/2 & 2 & 1/4 & 1
\end{pmatrix}.
\end{equation*}

To find the solution vector $\bm{x}=(x_{1},x_{2},x_{3},x_{4})^{T}$ by applying Theorem~\ref{T-minxAxxBx}, we have to start with evaluating the spectral radii $\mu$ and $\nu$, and constructing the functions $G$ and $H$. First, we calculate the matrix powers
\begin{equation*}
\bm{A}^{2}
=
\begin{pmatrix}
1 & 8 & 4 & 16
\\
1/3 & 1 & 4/3 & 2
\\
2 & 12 & 1 & 4
\\
1 & 3 & 2 & 1
\end{pmatrix},
\quad
\bm{A}^{3}
=
\begin{pmatrix}
8 & 48 & 4 & 16
\\
1 & 6 & 4/3 & 16/3
\\
4 & 12 & 8 & 4
\\
1 & 4 & 4 & 8
\end{pmatrix},
\quad
\bm{A}^{4}
=
\begin{pmatrix}
16 & 48 & 32 & 16
\\
8/3 & 16 & 4 & 16/3
\\
4 & 16 & 16 & 32
\\
4 & 24 & 4 & 16
\end{pmatrix},
\end{equation*}
and then use \eqref{E-lambda-ai1i2ai2i3aiki1} to obtain the spectral radius of $\bm{A}$ as follows: 
\begin{equation*}
\mu
=
\mathop\mathrm{tr}\bm{A}
\oplus
\mathop\mathrm{tr}\nolimits^{1/2}(\bm{A}^{2})
\oplus
\mathop\mathrm{tr}\nolimits^{1/3}(\bm{A}^{3})
\oplus
\mathop\mathrm{tr}\nolimits^{1/4}(\bm{A}^{4})
=
2.
\end{equation*}

In the same way, we form the matrices
\begin{equation*}
\bm{B}^{2}
=
\begin{pmatrix}
1 & 12 & 4 & 16
\\
1/2 & 1 & 2 & 4/3
\\
2 & 8 & 1 & 4
\\
1 & 2 & 2 & 1
\end{pmatrix},
\quad
\bm{B}^{3}
=
\begin{pmatrix}
8 & 32 & 4 & 16
\\
2/3 & 6 & 2 & 8
\\
4 & 8 & 8 & 4
\\
1 & 6 & 4 & 8
\end{pmatrix},
\quad
\bm{B}^{4}
=
\begin{pmatrix}
16 & 32 & 32 & 16
\\
4 & 16 & 8/3 & 8
\\
4 & 24 & 16 & 32
\\
4 & 16 & 4 & 16
\end{pmatrix},
\end{equation*}
to calculate their traces, and thus find the spectral radius of $\bm{B}$ to be
\begin{equation*}
\nu
=
2.
\end{equation*}
 
Furthermore, with $n=4$, the function $G$ defined by \eqref{E-Gs-Ht} takes the form
\begin{multline*}
G(s)
=
s^{-1}\mathop\mathrm{tr}(\bm{B}\bm{A})
\oplus
s^{-1/2}\mathop\mathrm{tr}\nolimits^{1/2}(\bm{B}\bm{A}^{2})
\oplus
s^{-1/3}\mathop\mathrm{tr}\nolimits^{1/3}(\bm{B}\bm{A}^{3})
\\
\oplus
s^{-2}\mathop\mathrm{tr}(\bm{B}^{2}\bm{A})
\oplus
s^{-1}\mathop\mathrm{tr}\nolimits^{1/2}(\bm{B}^{2}\bm{A}^{2})
\oplus
s^{-1}\mathop\mathrm{tr}\nolimits^{1/2}((\bm{B}\bm{A})^{2})
\oplus
s^{-3}\mathop\mathrm{tr}(\bm{B}^{3}\bm{A}).
\end{multline*}

To evaluate the coefficients in the function, we calculate the matrices
\begin{gather*}
\bm{B}\bm{A}
=
\begin{pmatrix}
1 & 8 & 4 & 16
\\
1/2 & 3/2 & 2 & 4/3
\\
2 & 12 & 3/2 & 4
\\
2/3 & 3 & 2 & 1
\end{pmatrix},
\qquad
\bm{B}\bm{A}^{2}
=
\begin{pmatrix}
8 & 48 & 4 & 16
\\
2/3 & 4 & 2 & 8
\\
4 & 12 & 8 & 6
\\
1 & 4 & 8/3 & 8
\end{pmatrix},
\\
\bm{B}\bm{A}^{3}
=
\begin{pmatrix}
16 & 48 & 32 & 16
\\
4 & 24 & 8/3 & 8
\\
4 & 18 & 16 & 32
\\
4 & 24 & 4 & 32/3
\end{pmatrix},
\qquad
\bm{B}^{2}\bm{A}
=
\begin{pmatrix}
8 & 48 & 6 & 16
\\
2/3 & 4 & 2 & 8
\\
8/3 & 12 & 8 & 4
\\
1 & 4 & 4 & 8
\end{pmatrix},
\\
\bm{B}^{2}\bm{A}^{2}
=
\begin{pmatrix}
16 & 48 & 32 & 24
\\
4 & 24 & 8/3 & 8
\\
4 & 16 & 32/3 & 32
\\
4 & 24 & 8/3 & 16
\end{pmatrix},
\qquad
(\bm{B}\bm{A})^{2}
=
\begin{pmatrix}
32/3 & 48 & 32 & 16
\\
4 & 24 & 3 & 8
\\
6 & 18 & 24 & 32
\\
4 & 24 & 6 & 32/3
\end{pmatrix},
\\
\bm{B}^{3}\bm{A}
=
\begin{pmatrix}
32/3 & 48 & 32 & 16
\\
4 & 24 & 3 & 8
\\
4 & 16 & 16 & 32
\\
4 & 24 & 4 & 16
\end{pmatrix}.
\end{gather*}

Evaluating the traces of the matrices obtained yields
\begin{gather*}
\mathop\mathrm{tr}(\bm{B}\bm{A})
=
3/2,
\qquad
\mathop\mathrm{tr}(\bm{B}\bm{A}^{2})
=
8,
\qquad
\mathop\mathrm{tr}(\bm{B}\bm{A}^{3})
=
24,
\qquad
\mathop\mathrm{tr}(\bm{B}^{2}\bm{A})
=
8,
\\
\mathop\mathrm{tr}(\bm{B}^{2}\bm{A}^{2})
=
24,
\qquad
\mathop\mathrm{tr}(\bm{B}\bm{A})^{2}
=
24,
\qquad
\mathop\mathrm{tr}(\bm{B}^{3}\bm{A})
=
24.
\end{gather*}

After substitution of traces and rearrangement of terms, the function becomes
\begin{equation*}
G(s)
=
24s^{-3}
\oplus
8s^{-2}
\oplus
24^{1/2}s^{-1}
\oplus
8^{1/2}s^{-1/2}
\oplus
24^{1/3}s^{-1/3}.
\end{equation*}
 
Similarly, we construct the function
\begin{multline*}
H(t)
=
t^{-1}\mathop\mathrm{tr}(\bm{B}\bm{A})
\oplus
t^{-2}\mathop\mathrm{tr}(\bm{B}\bm{A}^{2})
\oplus
t^{-3}\mathop\mathrm{tr}(\bm{B}\bm{A}^{3})
\oplus
t^{-1/2}\mathop\mathrm{tr}\nolimits^{1/2}(\bm{B}^{2}\bm{A})
\\
\oplus
t^{-1}\mathop\mathrm{tr}\nolimits^{1/2}(\bm{B}^{2}\bm{A}^{2})
\oplus
t^{-1}\mathop\mathrm{tr}\nolimits^{1/2}((\bm{B}\bm{A})^{2})
\oplus
t^{-1/3}\mathop\mathrm{tr}\nolimits^{1/3}(\bm{B}^{3}\bm{A}),
\end{multline*}
and then reduce it to
\begin{equation*}
H(t)
=
24t^{-3}
\oplus
8t^{-2}
\oplus
24^{1/2}t^{-1}
\oplus
8^{1/2}t^{-1/2}
\oplus
24^{1/3}t^{-1/3}.
\end{equation*}

We now construct the Pareto frontier for the problem as a set of points $(\alpha,\beta)$ given by \eqref{I-muleqalphaGnu-E-betaeqHalpha}. We start with the adjustment of the range of the parameter $\alpha$, defined by the inequality $\mu\leq\alpha\leq G(\nu)$. Since, the substitution $\nu=2$ yields $G(2)=3$, the inequality becomes $2\leq\alpha\leq3$.

Next, we consider the equality $\beta=H(\alpha)$, and refine the function $H(t)$ on the right-hand side by taking into account the range $2\leq t\leq3$. First, we note that the condition $t\geq2$ leads to the inequalities $8t^{-2}\leq24^{1/3}t^{-1/3}$ and $8^{1/2}t^{-1/2}\leq24^{1/3}t^{-1/3}$. At the same time, it follows from the inequality $t\leq3$ that $24^{1/2}t^{-1}\geq8^{1/2}t^{-1/2}$. As a result, if $2\leq t\leq 3$, the function reduces to $H(t)=24t^{-3}\oplus24^{1/2}t^{-1}\oplus24^{1/3}t^{-1/3}$.

We denote $\theta=24^{1/4}\approx2.2134$, and observe that, for all $t\geq\theta$, we have the inequality $24^{1/2}t^{-1}\leq24^{1/3}t^{-1/3}$, and for $t\leq\theta$, the inequality $24^{1/2}t^{-1}\leq24t^{-3}$. In this case, the function becomes $H(t)=24t^{-3}\oplus24^{1/3}t^{-1/3}$, which, together with the boundary condition on $\alpha$, yields the description of the Pareto frontier in the form
\begin{equation*}
2
\leq
\alpha
\leq
3,
\qquad
\beta
=
24\alpha^{-3}\oplus24^{1/3}\alpha^{-1/3}.
\end{equation*} 

With $\alpha$ and $\beta$ given by these conditions, and a vector $\bm{u}=(u_{1},u_{2},u_{3},u_{4})^{T}$, the Pareto-optimal solution to the problem is represented as
\begin{equation*}
\bm{x}
=
(\alpha^{-1}\bm{A}\oplus\beta^{-1}\bm{B})^{\ast}
\bm{u},
\qquad
\bm{u}
>
\bm{0}.
\end{equation*}

We conclude with the computation of two extreme solutions corresponding to the endpoints of the Pareto frontier, and an intermediate solution for an inner point of the frontier segment. After calculating $H(2)=3$ and $H(3)=2$, we represent the endpoints of the Pareto frontier segment as
\begin{equation*}
(\mu,H(\mu))
=
(2,3),
\qquad
(G(\nu),H(G(\nu)))
=
(3,2).
\end{equation*}

To obtain the matrix, which generates the solution under the conditions $\alpha=2$ and $\beta=3$, we take matrix
\begin{equation*}
2^{-1}\bm{A}\oplus3^{-1}\bm{B}
=
\begin{pmatrix}
1/2 & 3/2 & 2 & 1
\\
1/6 & 1/2 & 1/4 & 1/6
\\
1/8 & 1 & 1/2 & 2
\\
1/4 & 3/2 & 1/8 & 1/2
\end{pmatrix},
\end{equation*}
and calculate its second and third powers
\begin{equation*}
\begin{pmatrix}
1/4 & 2 & 1 & 4
\\
1/12 & 1/4 & 1/3 & 1/2
\\
1/2 & 3 & 1/4 & 1 
\\
1/4 & 3/4 & 1/2 & 1/4 
\end{pmatrix},
\qquad
\begin{pmatrix}
1 & 6 & 1/2 & 2
\\
1/8 & 3/4 & 1/6 & 2/3
\\
1/2 & 3/2 & 1 & 1/2
\\
1/8 & 1/2 & 1/2 & 1
\end{pmatrix}.
\end{equation*}

With these powers, we obtain the generating matrix in the form
\begin{equation*}
(2^{-1}\bm{A}\oplus3^{-1}\bm{B})^{\ast}
=
\begin{pmatrix}
1 & 6 & 2 & 4
\\
1/6 & 1 & 1/3 & 2/3
\\
1/2 & 3 & 1 & 2
\\
1/4 & 3/2 & 1/2 & 1
\end{pmatrix}.
\end{equation*}

It is easy to verify that all columns of this matrix are collinear to each other. Indeed, multiplications of the first column by $6$, $2$ and $4$ yield the second, third and fourth columns, respectively. Therefore, we can drop all columns except one, say the first, and write the solution, corresponding to the first endpoint of the Pareto frontier, as
\begin{equation*}
\bm{x}_{\mu}
=
\left(
\begin{array}{c}
1
\\
1/6
\\
1/2
\\
1/4
\end{array}
\right)
u,
\qquad
u>0.
\end{equation*}

By setting $u=1$, we have the vector of rates $\bm{x}_{\mu}\approx(1,0.1667,0.5,0.25)^{T}$.

In the similar way, we examine the second endpoint with $\alpha=3$ and $\beta=2$. We consider the matrix 
\begin{equation*}
3^{-1}\bm{A}\oplus2^{-1}\bm{B}
=
\begin{pmatrix}
1/2 & 1 & 2 & 1
\\
1/4 & 1/2 & 1/6 & 1/4
\\
1/8 & 3/2 & 1/2 & 2
\\
1/4 & 1 & 1/8 & 1/2
\end{pmatrix},
\end{equation*}
and calculate its powers to obtain the generating matrix
\begin{equation*}
(3^{-1}\bm{A}\oplus2^{-1}\bm{B})^{\ast}
=
\begin{pmatrix}
1 & 4 & 2 & 4
\\
1/4 & 1 & 1/2 & 1
\\
1/2 & 2 & 1 & 2
\\
1/4 & 1 & 1/2 & 1
\end{pmatrix}.
\end{equation*}

Since all columns in this matrix are collinear, we take the first column, and write the solution as
\begin{equation*}
\bm{x}_{\nu}
=
\begin{pmatrix}
1
\\
1/4
\\
1/2
\\
1/4
\end{pmatrix}
v,
\qquad
v>0.
\end{equation*}

With $v=1$, we have $\bm{x}_{\nu}=(1,0.25,0.5,0.25)^{T}$.

Finally, we assume that $\alpha=\theta$, where $\theta=24^{1/4}$, and note that $\beta=H(\theta)=\theta$. We form the matrix
\begin{equation*}
\theta^{-1}(\bm{A}\oplus\bm{B})
=
\theta^{-1}
\begin{pmatrix}
1 & 3 & 4 & 2
\\
1/2 & 1 & 1/2 & 1/2
\\
1/4 & 3 & 1 & 4
\\
1/2 & 3 & 1/4 & 1
\end{pmatrix},
\end{equation*}
and then find its second and third powers
\begin{equation*}
\theta^{-2}
\begin{pmatrix}
3/2 & 12 & 4 & 16
\\
1/2 & 3/2 & 2 & 2
\\
2 & 12 & 3/2 & 4
\\
3/2 & 3 & 2 & 3/2
\end{pmatrix},
\qquad
\theta^{-3}
\begin{pmatrix}
8 & 48 & 6 & 16
\\
1 & 6 & 2 & 8
\\
6 & 12 & 8 & 6
\\
3/2 & 6 & 6 & 8
\end{pmatrix}.
\end{equation*}

Calculation of the generating matrix yields
\begin{equation*}
(\theta^{-1}(\bm{A}\oplus\bm{B}))^{\ast}
=
\begin{pmatrix}
1 & 2\theta & 4\theta^{-1} & 2\theta^{2}/3
\\
\theta^{-1}/2 & 1 & \theta^{2}/12 & \theta/3
\\
\theta/4 & \theta^{2}/2 & 1 & 4\theta^{-1}
\\
\theta^{2}/16 & 3\theta^{-1} & \theta/4 & 1
\end{pmatrix}.
\end{equation*}

Since all columns are collinear, we take the first column to write the solution corresponding to $\alpha=\beta=\theta$, as 
\begin{equation*}
\bm{x}_{\theta}
=
\begin{pmatrix}
1
\\
\theta^{-1}/2
\\
\theta/4
\\
\theta^{2}/16
\end{pmatrix}
w,
\qquad
\theta 
=
24^{1/4},
\qquad
w>0.
\end{equation*}

If $w=1$, the solution can be represented as $\bm{x}_{\theta}\approx(1,0.2259,0.5533,0.3062)^{T}$.

Let us consider the obtained solutions
\begin{equation*}
x_{\mu}
\approx
\left(
\begin{array}{c}
1
\\
0.1667
\\
0.5
\\
0.25
\end{array}
\right),
\qquad
x_{\theta}
\approx
\left(
\begin{array}{c}
1
\\
0.2259
\\
0.5533
\\
0.3062
\end{array}
\right),
\qquad
x_{\nu}
=
\left(
\begin{array}{c}
1
\\
0.25
\\
0.5
\\
0.25
\end{array}
\right)
\end{equation*}

All of these solutions assign the highest rate to the first alternative. Next come the third and fourth alternatives. Finally, the second is rated lower than or equal to (as an extreme solution) the fourth alternative.
\end{example}

\section{Conclusions}

In this paper, we have developed an analytical solution to a bi-criteria decision problem to rate alternatives on the basis of pairwise comparisons according to two criteria. The problem was first formulated as a bi-objective optimization problem, where the objective functions are defined as the errors of the log-Chebyshev approximation of two symmetrically reciprocal matrices by a reciprocal matrix of unit rank. Then, we represented and solved the bi-objective problem in terms of a general idempotent semifield as a tropical optimization problem.

The solution approach is based on the introduction of two parameters that describe the optimal values of the objective function, and the reduction of the bi-objective problem to a parametrized vector inequality. The conditions for the existence of solutions to the inequality serve to describe the Pareto frontier for the optimization problem, whereas the corresponding solutions of the inequality act as the Pareto-optimal solution. We used this approach to derive a complete solution to the optimization problem in the form, which provides a direct description of both the Pareto frontier and corresponding Pareto-optimal solutions, and involves a polynomial computational complexity. 

We have applied the solution of the optimization problem to solve the bi-criteria decision problem of rating alternative in a compact vector form, which is ready for formal analysis and practical implementation. Examples of solving optimization and decision-making problems were given to illustrate the results obtained.
 
Possible lines of future research can include the application of tropical optimization to solve bi-criteria decision problems, where the criteria have different weights, and multi-criteria decision problems with equal and different weights.

\bibliographystyle{abbrvurl}

\bibliography{Using_tropical_optimization_techniques_in_bi-criteria_decision_problems}

\begin{thebibliography}{10}

\bibitem{Ahn2017Analytic}
B.~S. Ahn.
\newblock The analytic hierarchy process with interval preference statements.
\newblock {\em Omega}, 67:177--185, 2017.
\newblock \href {http://dx.doi.org/10.1016/j.omega.2016.05.004}
  {\path{doi:10.1016/j.omega.2016.05.004}}.

\bibitem{Barzilai1997Deriving}
J.~Barzilai.
\newblock Deriving weights from pairwise comparison matrices.
\newblock {\em J. Oper. Res. Soc.}, 48(12):1226--1232, 1997.
\newblock \href {http://dx.doi.org/10.2307/3010752}
  {\path{doi:10.2307/3010752}}.

\bibitem{Benson2009Multiobjective}
H.~P. Benson.
\newblock Multi-objective optimization: {P}areto optimal solutions, properties.
\newblock In C.~A. Floudas and P.~M. Pardalos, editors, {\em Encyclopedia of
  Optimization}, pages 2478--2481. Springer, 2 edition, 2009.
\newblock \href {http://dx.doi.org/10.1007/978-0-387-74759-0_426}
  {\path{doi:10.1007/978-0-387-74759-0_426}}.

\bibitem{Butkovic2010Maxlinear}
P.~Butkovi\v{c}.
\newblock {\em Max-linear Systems}.
\newblock Springer Monographs in Mathematics. Springer, London, 2010.
\newblock \href {http://dx.doi.org/10.1007/978-1-84996-299-5}
  {\path{doi:10.1007/978-1-84996-299-5}}.

\bibitem{Chu1998Ontheoptimal}
M.~T. Chu.
\newblock On the optimal consistent approximation to pairwise comparison
  matrices.
\newblock {\em Linear Algebra Appl.}, 272(1-3):155--168, 1998.
\newblock \href {http://dx.doi.org/10.1016/S0024-3795(97)00329-7}
  {\path{doi:10.1016/S0024-3795(97)00329-7}}.

\bibitem{Ehrgott2005Multicriteria}
M.~Ehrgott.
\newblock {\em Multicriteria Optimization}.
\newblock Springer, Berlin, 2 edition, 2005.
\newblock \href {http://dx.doi.org/10.1007/3-540-27659-9}
  {\path{doi:10.1007/3-540-27659-9}}.

\bibitem{Elsner2004Maxalgebra}
L.~Elsner and P.~{van den Driessche}.
\newblock Max-algebra and pairwise comparison matrices.
\newblock {\em Linear Algebra Appl.}, 385(1):47--62, 2004.
\newblock \href {http://dx.doi.org/10.1016/S0024-3795(03)00476-2}
  {\path{doi:10.1016/S0024-3795(03)00476-2}}.

\bibitem{Elsner2010Maxalgebra}
L.~Elsner and P.~{van den Driessche}.
\newblock Max-algebra and pairwise comparison matrices, {II}.
\newblock {\em Linear Algebra Appl.}, 432(4):927--935, 2010.
\newblock \href {http://dx.doi.org/10.1016/j.laa.2009.10.005}
  {\path{doi:10.1016/j.laa.2009.10.005}}.

\bibitem{Farkas2003Consistency}
A.~Farkas, P.~Lancaster, and P.~R\'{o}zsa.
\newblock Consistency adjustments for pairwise comparison matrices.
\newblock {\em Numer. Linear Algebra Appl.}, 10(8):689--700, 2003.
\newblock \href {http://dx.doi.org/10.1002/nla.318}
  {\path{doi:10.1002/nla.318}}.

\bibitem{Gavalec2015Decision}
M.~Gavalec, J.~Ram\'{\i}k, and K.~Zimmermann.
\newblock {\em Decision Making and Optimization}, volume 677 of {\em Lecture
  Notes in Economics and Mathematical Systems}.
\newblock Springer, Cham, 2015.
\newblock \href {http://dx.doi.org/10.1007/978-3-319-08323-0}
  {\path{doi:10.1007/978-3-319-08323-0}}.

\bibitem{Golan2003Semirings}
J.~S. Golan.
\newblock {\em Semirings and Affine Equations Over Them}, volume 556 of {\em
  Mathematics and Its Applications}.
\newblock Kluwer Acad. Publ., Dordrecht, 2003.
\newblock \href {http://dx.doi.org/10.1007/978-94-017-0383-3}
  {\path{doi:10.1007/978-94-017-0383-3}}.

\bibitem{Gondran2008Graphs}
M.~Gondran and M.~Minoux.
\newblock {\em Graphs, Dioids and Semirings}, volume~41 of {\em Operations
  Research/ Computer Science Interfaces}.
\newblock Springer, New York, 2008.
\newblock \href {http://dx.doi.org/10.1007/978-0-387-75450-5}
  {\path{doi:10.1007/978-0-387-75450-5}}.

\bibitem{Gursoy2013Analytic}
B.~B. Gursoy, O.~Mason, and S.~Sergeev.
\newblock The analytic hierarchy process, max algebra and multi-objective
  optimisation.
\newblock {\em Linear Algebra Appl.}, 438(7):2911--2928, 2013.
\newblock \href {http://dx.doi.org/10.1016/j.laa.2012.11.020}
  {\path{doi:10.1016/j.laa.2012.11.020}}.

\bibitem{Heidergott2006Maxplus}
B.~Heidergott, G.~J. Olsder, and J.~{van der Woude}.
\newblock {\em Max Plus at Work}.
\newblock Princeton Series in Applied Mathematics. Princeton Univ. Press,
  Princeton, NJ, 2006.

\bibitem{Krivulin2014Tropical}
N.~Krivulin.
\newblock Tropical optimization problems.
\newblock In L.~A. Petrosyan, J.~V. Romanovsky, and D.~W.~K. Yeung, editors,
  {\em Advances in Economics and Optimization}, Economic Issues, Problems and
  Perspectives, pages 195--214. Nova Sci. Publ., New York, 2014.
\newblock \href {http://arxiv.org/abs/1408.0313} {\path{arXiv:1408.0313}}.

\bibitem{Krivulin2015Extremal}
N.~Krivulin.
\newblock Extremal properties of tropical eigenvalues and solutions to tropical
  optimization problems.
\newblock {\em Linear Algebra Appl.}, 468:211--232, 2015.
\newblock \href {http://arxiv.org/abs/1311.0442} {\path{arXiv:1311.0442}},
  \href {http://dx.doi.org/10.1016/j.laa.2014.06.044}
  {\path{doi:10.1016/j.laa.2014.06.044}}.

\bibitem{Krivulin2015Multidimensional}
N.~Krivulin.
\newblock A multidimensional tropical optimization problem with nonlinear
  objective function and linear constraints.
\newblock {\em Optimization}, 64(5):1107--1129, 2015.
\newblock \href {http://dx.doi.org/10.1080/02331934.2013.840624}
  {\path{doi:10.1080/02331934.2013.840624}}.

\bibitem{Krivulin2015Rating}
N.~Krivulin.
\newblock Rating alternatives from pairwise comparisons by solving tropical
  optimization problems.
\newblock In Z.~Tang, J.~Du, S.~Yin, L.~He, and R.~Li, editors, {\em 2015 12th
  Intern. Conf. on Fuzzy Systems and Knowledge Discovery (FSKD)}, pages
  162--167. IEEE, 2015.
\newblock \href {http://arxiv.org/abs/1504.00800} {\path{arXiv:1504.00800}},
  \href {http://dx.doi.org/10.1109/FSKD.2015.7381933}
  {\path{doi:10.1109/FSKD.2015.7381933}}.

\bibitem{Krivulin2016Using}
N.~Krivulin.
\newblock Using tropical optimization techniques to evaluate alternatives via
  pairwise comparisons.
\newblock In A.~H. Gebremedhin, E.~G. Boman, and B.~Ucar, editors, {\em 2016
  Proc. 7th SIAM Workshop on Combinatorial Scientific Computing}, pages 62--72.
  SIAM, Philadelphia, PA, 2016.
\newblock \href {http://arxiv.org/abs/1503.04003} {\path{arXiv:1503.04003}},
  \href {http://dx.doi.org/10.1137/1.9781611974690.ch7}
  {\path{doi:10.1137/1.9781611974690.ch7}}.

\bibitem{Krivulin2017Direct}
N.~Krivulin.
\newblock Direct solution to constrained tropical optimization problems with
  application to project scheduling.
\newblock {\em Comput. Manag. Sci.}, 14(1):91--113, 2017.
\newblock \href {http://arxiv.org/abs/1501.07591} {\path{arXiv:1501.07591}},
  \href {http://dx.doi.org/10.1007/s10287-016-0259-0}
  {\path{doi:10.1007/s10287-016-0259-0}}.

\bibitem{Krivulin2017Tropical}
N.~Krivulin.
\newblock Tropical optimization problems with application to project scheduling
  with minimum makespan.
\newblock {\em Ann. Oper. Res.}, 256(1):75--92, 2017.
\newblock \href {http://arxiv.org/abs/1403.0268} {\path{arXiv:1403.0268}},
  \href {http://dx.doi.org/10.1007/s10479-015-1939-9}
  {\path{doi:10.1007/s10479-015-1939-9}}.

\bibitem{Krivulin2017Using}
N.~Krivulin.
\newblock Using tropical optimization to solve constrained minimax
  single-facility location problems with rectilinear distance.
\newblock {\em Comput. Manag. Sci.}, 14(4):493--518, 2017.
\newblock \href {http://arxiv.org/abs/1511.07549} {\path{arXiv:1511.07549}},
  \href {http://dx.doi.org/10.1007/s10287-017-0289-2}
  {\path{doi:10.1007/s10287-017-0289-2}}.

\bibitem{Krivulin2018Methods}
N.~Krivulin.
\newblock Methods of tropical optimization in rating alternatives based on
  pairwise comparisons.
\newblock In A.~Fink, A.~F{\"u}genschuh, and M.~J. Geiger, editors, {\em
  Operations Research Proceedings 2016}, Operations Research Proceedings, pages
  85--91. Springer, Cham, 2018.
\newblock \href {http://arxiv.org/abs/1608.02666} {\path{arXiv:1608.02666}},
  \href {http://dx.doi.org/10.1007/978-3-319-55702-1_13}
  {\path{doi:10.1007/978-3-319-55702-1_13}}.

\bibitem{Krivulin2017Tropicaloptimization}
N.~Krivulin and S.~Sergeev.
\newblock Tropical optimization techniques in multi-criteria decision making
  with {A}nalytical {H}ierarchy {P}rocess.
\newblock In D.~Al-Dabass, Z.~Xie, A.~Orsoni, and A.~Pantelous, editors, {\em
  UKSim-AMSS 11th European Modelling Symposium on Computer Modelling and
  Simulation (EMS 2017)}, pages 38--43. IEEE, 2017.
\newblock \href {http://arxiv.org/abs/1801.10524} {\path{arXiv:1801.10524}},
  \href {http://dx.doi.org/10.1109/EMS.2017.18}
  {\path{doi:10.1109/EMS.2017.18}}.

\bibitem{Kubler2016Stateoftheart}
S.~Kubler, J.~Robert, W.~Derigent, A.~Voisin, and Y.~L. Traon.
\newblock A state-of the-art survey and testbed of fuzzy {AHP} ({FAHP})
  applications.
\newblock {\em Expert Syst. Appl.}, 65:398--422, 2016.
\newblock \href {http://dx.doi.org/https://doi.org/10.1016/j.eswa.2016.08.064}
  {\path{doi:https://doi.org/10.1016/j.eswa.2016.08.064}}.

\bibitem{Luc2008Pareto}
D.~T. Luc.
\newblock Pareto optimality.
\newblock In A.~Chinchuluun, P.~M. Pardalos, A.~Migdalas, and L.~Pitsoulis,
  editors, {\em Pareto Optimality, Game Theory and Equilibria}, pages 481--515.
  Springer, New York, 2008.
\newblock \href {http://dx.doi.org/10.1007/978-0-387-77247-9_18}
  {\path{doi:10.1007/978-0-387-77247-9_18}}.

\bibitem{Maclagan2015Introduction}
D.~Maclagan and B.~Sturmfels.
\newblock {\em Introduction to Tropical Geometry}, volume 161 of {\em Graduate
  Studies in Mathematics}.
\newblock AMS, Providence, RI, 2015.

\bibitem{Mceneaney2006Maxplus}
W.~M. McEneaney.
\newblock {\em Max-Plus Methods for Nonlinear Control and Estimation}.
\newblock Systems and Control: Foundations and Applications. Birkh\"{a}user,
  Boston, 2006.
\newblock \href {http://dx.doi.org/10.1007/0-8176-4453-9}
  {\path{doi:10.1007/0-8176-4453-9}}.

\bibitem{Pappalardo2008Multiobjective}
M.~Pappalardo.
\newblock Multiobjective optimization: a brief overview.
\newblock In A.~Chinchuluun, P.~M. Pardalos, A.~Migdalas, and L.~Pitsoulis,
  editors, {\em Pareto Optimality, Game Theory and Equilibria}, pages 517--528.
  Springer, New York, 2008.
\newblock \href {http://dx.doi.org/10.1007/978-0-387-77247-9_19}
  {\path{doi:10.1007/978-0-387-77247-9_19}}.

\bibitem{Saaty1977Scaling}
T.~L. Saaty.
\newblock A scaling method for priorities in hierarchical structures.
\newblock {\em J. Math. Psych.}, 15(3):234--281, 1977.
\newblock \href {http://dx.doi.org/10.1016/0022-2496(77)90033-5}
  {\path{doi:10.1016/0022-2496(77)90033-5}}.

\bibitem{Saaty1990Analytic}
T.~L. Saaty.
\newblock {\em The Analytic Hierarchy Process}.
\newblock RWS Publications, Pittsburgh, PA, 2 edition, 1990.

\bibitem{Saaty2013Onthemeasurement}
T.~L. Saaty.
\newblock On the measurement of intangibles: A principal eigenvector approach
  to relative measurement derived from paired comparisons.
\newblock {\em Notices Amer. Math. Soc.}, 60(2):192--208, 2013.
\newblock \href {http://dx.doi.org/10.1090/noti944}
  {\path{doi:10.1090/noti944}}.

\bibitem{Saaty1984Comparison}
T.~L. Saaty and L.~G. Vargas.
\newblock Comparison of eigenvalue, logarithmic least squares and least squares
  methods in estimating ratios.
\newblock {\em Math. Modelling}, 5(5):309--324, 1984.
\newblock \href {http://dx.doi.org/10.1016/0270-0255(84)90008-3}
  {\path{doi:10.1016/0270-0255(84)90008-3}}.

\bibitem{Thurstone1927Law}
L.~L. Thurstone.
\newblock A law of comparative judgment.
\newblock {\em Psychological Review}, 34(4):273--286, 1927.
\newblock \href {http://dx.doi.org/10.1037/h0070288}
  {\path{doi:10.1037/h0070288}}.

\bibitem{Tran2013Pairwise}
N.~M. Tran.
\newblock Pairwise ranking: Choice of method can produce arbitrarily different
  rank order.
\newblock {\em Linear Algebra Appl.}, 438(3):1012--1024, 2013.
\newblock \href {http://dx.doi.org/10.1016/j.laa.2012.08.028}
  {\path{doi:10.1016/j.laa.2012.08.028}}.

\bibitem{Laarhoven1983Fuzzy}
P.~J.~M. {van Laarhoven} and W.~Pedrycz.
\newblock A fuzzy extension of {S}aaty's priority theory.
\newblock {\em Fuzzy Sets and Systems}, 11(1):229--241, 1983.
\newblock \href {http://dx.doi.org/10.1016/S0165-0114(83)80082-7}
  {\path{doi:10.1016/S0165-0114(83)80082-7}}.

\end{thebibliography}

\end{document}